\documentclass[11pt,reqno]{amsart}
\usepackage{amsthm,amssymb,amsmath}
\usepackage{textpos}

\newtheorem{theorem}{Theorem}
\newtheorem{lemma}{Lemma}
\newtheorem{proposition}{Proposition}

\theoremstyle{definition}

\newtheorem{definition}{\sc Definition}
\newtheorem*{definition*}{\sc Definition}

\newtheorem*{examples}{Examples}
\newtheorem{remark}{\bf Remark}

\newtheorem*{remark*}{\bf Remark}
\newtheorem*{example*}{\bf Example}

\newcommand{\loc}{{\rm loc}}

\expandafter\def\expandafter\normalsize\expandafter{%
    \normalsize
    \setlength\abovedisplayshortskip{8pt}
    \setlength\belowdisplayshortskip{8pt}
}

\begin{document}

\title[Stochastic transport equation]{Stochastic transport equation with singular drift}

\author{Damir Kinzebulatov}

\address{D\'{e}partement de math\'{e}matiques et de statistique, Universit\'{e} Laval, Qu\'{e}bec, QC, G1V 0A6, Canada}

\email{damir.kinzebulatov@mat.ulaval.ca}

\author{Yuliy A. Sem\"{e}nov}

\address{University of Toronto, Department of Mathematics, Toronto, ON, M5S 2E4, Canada}

\email{semenov.yu.a@gmail.com}

\author{Renming Song}

\address{Department of Mathematics, University of Illinois, Urbana, IL 61801, USA}

\email{rsong@math.uiuc.edu}

%\dedicatory{\today}

\begin{abstract}
We prove existence, uniqueness and Sobolev regularity of weak solution of 
the Cauchy problem of the stochastic transport equation with
drift in a large class of singular vector fields containing, in particular, the $L^d$ class, the weak $L^d$ class, as well as some vector fields that are not even  in 
$L_{\loc}^{2+\varepsilon}$ for any $\varepsilon>0$.
\end{abstract}


\thanks{The research of D. K. is supported by grants from NSERC and FRQNT.
The research of R. S. is supported in part by a grant from the Simons Foundation (\#429343).}

\maketitle

\medskip

\section{Introduction}

\label{intro_sect}

Throughout this paper we assume $d\ge 3$. Let $B_t$ be a Brownian motion in $\mathbb R^d$ 
defined on a probability space $(\Omega, \mathcal F,\mathbb P)$ 
with respect to a complete and right-continuous filtration $\mathcal F_t$. 
Let $\circ$ denote the Stratonovich multiplication. Set 
$L^p \equiv L^p(\mathbb R^d) \equiv L^p(\mathbb R^d,dx)$, 
$L^p_{\loc} \equiv L^p_{\loc}(\mathbb R^d), W^{1, p}\equiv W^{1, p}(\mathbb R^d), W^{1, p}_{\loc}\equiv W^{1, p}_{\loc}(\mathbb R^d), C^\infty_c\equiv C^\infty_c(\mathbb R^d)$.
We denote by $\|\cdot\|_{p \rightarrow q}$ the operator norm $\|\cdot\|_{L^p \rightarrow L^q}$.

\medskip

The subject of this paper is the problem of existence, uniqueness and Sobolev regularity of weak solution to the Cauchy problem for the stochastic transport equation (STE)
\begin{equation}
\label{eq1}
\begin{array}{c}
du+b \cdot\nabla u dt + \sigma \nabla u \circ dB_t=0 \quad \text{ on } (0,\infty) \times \mathbb R^d, 
\\[3mm]
u|_{t=0} = f ,
\end{array}
\end{equation}
where $u(t,x)$ is a scalar random field, $\sigma \neq 0$, 
$f$ is in $L^p$ or $W^{1, p}$, 
and $b:\mathbb R^d \rightarrow \mathbb R^d$ is in the class of \textit{form-bounded} vector fields (see definition below), a large class of singular vector fields containing, in particular, 
vector fields $b$ with $|b| \in L^d$,  or with $|b|$ in the weak $L^d$ class, 
as well as some vector fields $b$ with $|b| \not \in L_{\loc}^{2+\varepsilon}$
for any $\varepsilon>0$. 

It is well known that the Cauchy problem for the deterministic transport equation $\partial_t u + b \cdot \nabla u=0$ (corresponding to $\sigma=0$ in \eqref{eq1}) is in general not well posed already for a bounded but discontinuous $b$. Moreover, in that case, even if the initial function $f$ is regular, one can not hope that 
the corresponding solution $u$ will be regular 
immediately after $t=0$.
This, however, changes if one adds the noise term $\sigma \nabla u \circ dB_t$, $\sigma>0$. 
For the stochastic STE \eqref{eq1}, a unique weak solution exists and is regular for
some discontinuous $b$. 
This effect of regularization and well-posedness by noise, demonstrated by the STE, attracted considerable interest in the past few years, as a part of the more general program of establishing well-posedness by noise for SPDEs whose deterministic counterparts arising in fluid dynamics are not well-posed, see \cite{BFGM, GM} 
for detailed discussions and further references.

In \cite{BFGM}, the authors establish existence, uniqueness and Sobolev $W^{1,p}$-regularity 
(up to the initial time $t=0$, with $p$ large)
for weak solutions of \eqref{eq1} with time-dependent drift $b$ satisfying
$$
|b(\cdot,\cdot)| \in L^q\bigl([0,\infty),L^r+L^\infty\bigr), \quad \frac{d}{r}+\frac{2}{q} \leqslant 1
$$
(actually, \cite{BFGM} allows $b=b_1+b_1$ with $b_1$ satisfying the condition above and $b_2$ being continuously differentiable with at most linear growth at infinity; their uniqueness result imposes additional assumptions on ${\rm div\,}b$).
They apply this result to study the SDE
\begin{equation}
\label{sde0}
X_t=x-\int_s^t b(r,X_r)dr+\sigma (B_t-B_s),
\end{equation}
constructing, in particular, a unique, $W^{1,p}$-regular stochastic Lagrangian flow that solves \eqref{sde0}
for a.e.\,$x \in \mathbb R^d$. The STE can be viewed as the equation behind both the SDE (via path-wise interpretation of the STE and the SDE, see \cite{BFGM}) and the parabolic equation $(\partial_t - \frac{\sigma^2}{2}\Delta + b \cdot \nabla)v=0$ (arising from \eqref{eq1} upon taking expectation, i.e.\,$v=\mathbb E[u]$, see, if needed, \eqref{eq_ito} below).

In this paper, we show that 
the regularity and well-posedness for \eqref{eq1} hold for 
a much larger class of drifts $b$,
at least in the time-independent case $b=b(x)$ (see, however, Remark \ref{time_rem} below concerning time-dependent $b$). 

\begin{definition}
A Borel vector field $b:\mathbb R^d \rightarrow \mathbb R^d$ is said to be 
form-bounded with relative bound $\delta>0$, written as $b\in \mathbf{F}_\delta$,
if $|b| \in L^2_{\loc}$ and there exists a constant $\lambda=\lambda_\delta \geq 0$ such that
\begin{equation*}
\||b|(\lambda -\Delta)^{-\frac{1}{2}}\|_{2\rightarrow 2}\leq \sqrt{\delta}.
\end{equation*}
\end{definition}

It is easily seen that the condition $b \in \mathbf{F}_\delta$ can be stated equivalently 
as a quadratic form inequality 
$$
\| b\varphi\|_2^2 \leq \delta \| \nabla\varphi\|_2^2 + c_\delta \|\varphi\|_2^2, \quad \varphi 
\in W^{1,2},
$$
for a constant $c_\delta\,(=\lambda \delta)$. 
Let us also note that
$$
b_1 \in \mathbf{F}_{\delta_1}, b_2 \in \mathbf{F}_{\delta_2} \quad \Rightarrow \quad b_1+b_2 \in \mathbf{F}_\delta, \qquad \sqrt{\delta}:=\sqrt{\delta_1}+\sqrt{\delta_2}.
$$

\begin{examples}
1.~Any vector field $$b \in L^d(\mathbb R^d,\mathbb R^d) + L^\infty(\mathbb R^d,\mathbb R^d)$$ 
is in $\mathbf{F}_\delta$ for $\delta>0$ that can be chosen arbitrarily small. 
Indeed, for any $\varepsilon>0$ we can write $b=\mathsf{f}+\mathsf{h}$ with $\|\mathsf{f}\|_d<\varepsilon$, $\mathsf{h} \in L^\infty(\mathbb R^d,\mathbb R^d)$. It follows from H\"{o}lder's inequality and the Sobolev embedding theorem that
for any $g \in L^2$,
\begin{align*}
\||b|(\lambda-\Delta)^{-\frac{1}{2}}g\|_2 & \leq \|\mathsf{f}\|_d \|(\lambda-\Delta)^{-\frac{1}{2}}g\|_{\frac{2d}{d-2}} + \|\mathsf{h}\|_\infty \lambda^{-\frac{1}{2}}\|g\|_2 
\\
& \leq c\|\mathsf{f}\|_d \|g\|_2 + \|\mathsf{h}\|_\infty \lambda^{-\frac{1}{2}}\|g\|_2  \leq (c+1)\varepsilon \|g\|_2 \quad \text{ for } \lambda=\varepsilon^{-2}\|\mathsf{h}\|^{-2}_\infty.
\end{align*}

2.~The class $\mathbf{F}_\delta$ also contains vector fields having critical-order singularities, such as $$b(x)=\pm\sqrt{\delta}\frac{d-2}{2}|x|^{-2}x$$ (by Hardy's inequality $\frac{(d-2)^2}{4}\||x|^{-1}\varphi\|_2^2 \leq \|\nabla \varphi\|_2^2$, $\varphi \in W^{1,2}$).

3.~More generally, the class $\mathbf{F}_\delta$ contains vector fields
$b$ with $|b|$ in $L^{d,w}$ (the weak $L^d$ space). Recall that a Borel function $h:\mathbb R^d \rightarrow \mathbb R$ is in $L^{d,w}$ if $$\|h\|_{d,w}:=\sup_{s>0}s|\{x \in \mathbb R^d: |h(x)|>s\}|^{1/d}<\infty.$$ By the Strichartz inequality with sharp constant \cite[Prop.~2.5, 2.6, Cor.~2.9]{KPS}, if $|b|$ in $L^{d,w}$, then 
$b \in \mathbf{F}_{\delta_1}$ with 
\begin{align*}
\sqrt{\delta_1}&=\||b| (\lambda - \Delta)^{-\frac{1}{2}} \|_{2 \rightarrow 2} \\ & \leq
\|b\|_{d,w} \Omega_d^{-\frac{1}{d}} \||x|^{-1} (\lambda - \Delta)^{-\frac{1}{2}} \|_{2 \rightarrow 2} \\ & \leq \|b\|_{d,w} \Omega_d^{-\frac{1}{d}} \frac{2}{d-2},
\end{align*}
where $\Omega_d=\pi^{\frac{d}{2}}\Gamma(\frac{d}{2}+1)$ is the volume of the unit ball in $\mathbb R^d$. 

We also note that if $h \in L^2(\mathbb R)$, $T:\mathbb R^d \rightarrow \mathbb R$ is a linear map, then the vector field $b(x)=h(Tx)e$, where $e \in \mathbb R^d$, is in $\mathbf{F}_{\delta}$ with appropriate $\delta$, but $|b|$ may not be in $L^{d,w}_{\loc}$.

4.~More generally, the class $\mathbf{F}_\delta$ contains vector fields in the Campanato-Morrey class and the Chang-Wilson-Wolff class, with $\delta$ depending on the respective norms of the vector field in these classes, see \cite{CWW}.

5.~We note that 
there exists $b \in \mathbf{F}_\delta$ such that $|b| \not\in L^{2+\varepsilon}_{\loc}(\mathbb R^d,\mathbb R^d)$ for any $\varepsilon>0$,
e.g., consider
$$
|b(x)|^2=C\frac{\mathbf{1}_{B(0,1+\alpha)} - \mathbf{1}_{B(0,1-\alpha)}}{\big| |x|-1\big|^{-1}(-\ln\big||x|-1\big|)^\beta}, \quad \beta>1, \quad 0<\alpha<1.
$$

We emphasize that the condition $b \in \mathbf{F}_\delta$ is not a refinement of $|b| \in L^d+L^\infty$ in the sense that $\mathbf{F}_\delta$ is not situated between $L^d+L^\infty$ and $L^p+L^\infty$, $p<d$. In contrast to the elementary sub-classes of $\mathbf{F}_\delta$ listed above, the class $\mathbf{F}_\delta$ is defined in terms of the operators that, essentially, constitute the equation in \eqref{eq1}.
\end{examples}

The key result of this paper is the Sobolev regularity of 
solutions $u$ to the Cauchy problem for the STE \eqref{eq1}:
\begin{equation}
\label{reg0}
\sup_{t \in [0,T]}\bigl\|\mathbb E|\nabla u|^{2q}\bigr\|_2 \leq C\|\nabla f\|^{2q}_{4q}, \quad q=1,2,\dots,
\end{equation}
provided that $b$ is in $\mathbf{F}_\delta$ with $\delta$ smaller than a certain explicit constant, see Theorem \ref{thm1_reg}. This is a stochastic (parabolic) counterpart of the Sobolev regularity 
%estiamtes for solutions of the corresponding deterministic elliptic equation 
estimates for solutions of the corresponding deterministic elliptic equation 
established in \cite{KS}.
More precisely, in \cite{KS} the authors consider the 
operator $-\Delta + b \cdot \nabla$, $b \in \mathbf{F}_\delta$ with $0<\delta<1 \wedge \bigl(\frac{2}{d-2}\bigr)^2$, $d \geq 3$ and 
establish the following Sobolev regularity of solutions 
$v$ to the elliptic equation $(\mu-\Delta +b \cdot \nabla)v=f$ in $L^q$ for $2 \vee (d-2)\leq q < \frac{2}{\sqrt{\delta}}$:
\begin{equation}
\label{reg}
\begin{array}{c}
\|\nabla v \|_{\frac{qd}{d-2}} \leq  K \|f\|_q, 
\end{array}
\end{equation}
with $K$ depending only on $d$, $q$, the relative bound 
$\delta$ and $c_\delta$. 
The estimate \eqref{reg} is needed in \cite{KS} to run a Moser-type  iteration procedure that yields the 
Feller semigroup corresponding to $-\Delta +b \cdot \nabla$.
It was established in \cite{KiS2}
that, given $b \in \mathbf{F}_\delta$ with $\delta<1 \wedge \bigl(\frac{2}{d-2}\bigr)^2$, this Feller semigroup determines, for every starting point $x \in \mathbb R^d$, a weak solution to the SDE 
\begin{equation}
\label{sde0_}
X_t=x-\int_0^t b(X_r)dr + \sqrt{2}B_t
\end{equation}
 (see also \cite{KiS} where the authors consider drifts in a larger class).

The approach to studying SDEs via regularity theory of the STE, developed in \cite{BFGM}, can be combined with Theorem \ref{thm1_reg}
to obtain strong existence and uniqueness for \eqref{sde0} with $b \in \mathbf{F}_\delta$ (cf.\,Remark \ref{rem_sdes} below), albeit potentially excluding a measure zero set of starting points $x \in \mathbb R^d$. 
For results on strong existence and uniqueness 
for any $x \in \mathbb R^d$, with $b$ satisfying (in the time-independent case) $|b| \in L^p+L^\infty$ with $p>d$ or $p=d$, see \cite{Kr1, Kr2, KrR}.

\medskip

We conclude this introduction with 
a few remarks concerning the criticality 
of the singularities of form-bounded drifts.

\smallskip

1.~In \cite[Sect.\,7]{BFGM}, the authors show that the SDE \eqref{sde0_} with drift $b(x)=\beta |x|^{-2}x$ and starting point $x=0$ does not have a weak solution if $\beta>d-2$. 
In view of Example 2 above, 
this drift $b$ belongs to $\mathbf{F}_\delta$ with $\sqrt{\delta}=\beta \frac{2}{d-2}$, so by the result of \cite{KiS2} cited above, the weak solution to \eqref{sde0_} with $x=0$ exists as long as $\beta>0$ satisfies $\beta < \frac{1}{2}$ if $d=3$, $\beta<1$ if $d \geq 4$ (in fact, for $d \geq 5$ it suffices to require $\beta<\frac{d-3}{2}$ using \cite[Corollary 4.10]{KiS3}). Thus, the weak well-posedness of \eqref{sde0_} is sensitive to changes in 
the value of the constant multiple $\beta$ of $b$ (equivalently, changes in the value of the relative bound $\delta$). 
In this sense, the singularities of $b \in \mathbf{F}_\delta$ are critical.

Let us note that the diffusion process with  drift $b(x)=c |x|^{-2}x$, $c \in \mathbb R$, was studied earlier in \cite{W}.

\smallskip

2. Let $b \in \mathbf{F}_\delta$. There is a quantitative dependence between the value of the relative bound $\delta$ and the regularity properties of solutions to the corresponding equations (PDEs or STEs). Indeed, the admissible values of $q$ in \eqref{reg}, as well as in \eqref{reg0}, depend on the value of $\delta$. This dependence is lost if one considers $b$ with $|b| \in L^d + L^\infty$ since any such $b$ has arbitrarily small relative bound, cf.\,Example 1.

\smallskip

3. Concerning the difference between classes $\mathbf{F}_\delta$ and its subclass $L^d+L^\infty$, let us also note the following: if $v$ is a weak solution of the elliptic equation $(\lambda -\Delta  + b \cdot \nabla )v=f$, $\lambda>0$, $f \in C_c^\infty$ with $|b| \in L^d+L^\infty$ and $v \in W^{1,r}$ for $r$ large (e.g.\,by \eqref{reg}), then, by H\"{o}lder's inequality,
$$
\Delta v \in L^{\frac{rd}{d+r}}_{\loc}. 
$$ 
However, for $b \in \mathbf{F}_\delta$, one can only say that (cf.\,Example 5 above)  $$\Delta v \in L^{\frac{2d}{d+2}}_{\loc} $$ (one can in fact show that $v \in W^{2,2}$). 
That is, in case $b \in \mathbf{F}_\delta$, there are no $W^{2,p}$ estimates on solution $v$ for $p$ large.

See \cite{KiS3} for detailed discussions of 
remarks 2 and 3 above.

\subsection*{Notations}
Denote 
$$
\langle f,g\rangle = \langle f g\rangle :=\int_{\mathbb R^d}f gdx
$$ 
(all functions considered below are assumed to be real-valued).

Set
\begin{align*}
\rho(x)  \equiv \rho_{\kappa, \theta}(x) 
:=(1+\kappa |x|^2)^{-\theta}, \quad \kappa>0,  \quad \theta>\frac{d}{2}, \quad x \in \mathbb R^d.
\end{align*}
It is easily seen that 
\begin{equation}
\label{two_est}
|\nabla \rho(x)| \leq \theta\sqrt{\kappa}\rho(x), \quad x \in \mathbb R^d.
\end{equation}
Below we will be applying \eqref{two_est} to $\rho$ with $\kappa$ chosen sufficiently small.

For any $p>1$, we use $p'$ to denote its conjugate $p/(p-1)$. Let $L^p_\rho \equiv L^p(\mathbb R^d,\rho dx)$. 
Denote by $\|\cdot\|_{p,\rho}$ the norm in $L^p_\rho$, and by $\langle \cdot ,\cdot \rangle_\rho$ the inner product in $L^2_\rho$.

Set $W^{1,2}_\rho:=
\{g \in W^{1,2}_{\loc} \mid 
\|g\|_{W^{1,2}_\rho}:=\|g\|_{2,\rho}+\|\nabla g\|_{2,\rho}<\infty\}$.

Define constants $$\beta_{2q}:=1+4qd, \quad q=1,2,\dots$$

Put $J_T := [0,T]$. 

\section{Main results}

Below we consider the Cauchy problem for the STE
\begin{equation}
\label{eq2}
\tag{CP}
\begin{array}{c}
du  + \mu\, u dt  +b \cdot\nabla u dt+ \sigma \nabla u \circ dB_t=0 \quad \text{ on } (0,\infty) \times \mathbb R^d, \\[3mm]
u|_{t=0} = f \in L^p, \quad p \geq 2,
\end{array}
\end{equation}
where $\mu \geq 0$. Since solutions of the Cauchy problems \eqref{eq1} and \eqref{eq2} will differ by a multiple $e^{-\mu t}$, it suffices to prove the well-posedness of \eqref{eq2}.

Let us first make a few preliminary remarks.

1.~We can rewrite the equation in \eqref{eq2}, using the identity relating Stratonovich and It\^{o} integrals
\begin{equation}
\label{strat_id}
\int_0^t \nabla u \circ dB_s=\int_0^t \nabla u dB_s-\frac{1}{2}\sum_{k=1}^d[\partial_{x_k}u,B^k]_t, \qquad B_t=(B_t^k)_{k=1}^d,
\end{equation}
as
\begin{equation}
\label{eq_ito}
du + \mu u dt+b \cdot\nabla u dt + \sigma \nabla u  dB_t-\frac{\sigma^2}{2}\Delta u=0.
\end{equation}

\smallskip

2.~If 
$b \in C_c^\infty(\mathbb R^d,\mathbb R^d)$ and 
$f \in C_c^\infty$, 
then (see 
\cite[Theorem 6.1.9]{Ku}) 
there exists a unique adapted strong solution of \eqref{eq2}
\begin{equation*}
u(t) - f+ \mu \int_0^t u ds +\int_0^t b \cdot \nabla uds + \sigma\int_0^t \nabla u \circ dB_s=0 \text{ a.s.},
 \quad t \in J_T,
\end{equation*}
given by
\begin{equation}
e^{-\mu t }u(t)=f(\Psi_{t}^{-1}), \quad t \geqslant 0,
\end{equation}
where $\Psi_{t}:\mathbb R^d \times \Omega \rightarrow \mathbb R^d$ is the stochastic flow for the SDE 
\begin{equation}
\label{sde}
X_t=x-\int_0^t b(X_r)dr+\sigma B_t,
\end{equation}
i.e.~there exists $\Omega_0 \subset \Omega$, $\mathbb P(\Omega_0)=1$, such that, for all $\omega \in \Omega_0$,
$\Psi_{t}(\cdot,\omega) \Psi_{s}(\cdot,\omega) = \Psi_{t+s}(\cdot,\omega)$, $\Psi_{0}(x,\omega)=x$, and

1) for every $x \in \mathbb R^d$, the process $t \mapsto \Psi_{t}(x,\omega)$ is a strong solution of \eqref{sde},

2) $\Psi_{t}(x,\omega)$ is continuous in $(t,x)$, $\Psi_{t}(\cdot,\omega):\mathbb R^d \rightarrow \mathbb R^d$ are homeomorphisms,  and $\Psi_{t}(\cdot,\omega)$, $\Psi_{t}^{-1}(\cdot,\omega) \in C^\infty(\mathbb R^d,\mathbb R^d)$.

\medskip

We first state our basic existence result. Recall that $b \in \mathbf{F}_\delta$ if 
$$
\| b\varphi\|_2^2 \leq \delta \| \nabla\varphi\|_2^2 + c_\delta \|\varphi\|_2^2, \quad \varphi \in W^{1,2},
$$
for some constant $c_\delta \geq 0$.

\begin{theorem}
\label{thm1}
Assume that $d \geq 3$, $b \in \mathbf{F}_\delta$ with
$\sqrt{\delta}<\frac{\sigma^2}{2\beta_2}$.
Let $T>0$, $p \geq 2$. 
Provided that $\kappa$ is chosen sufficiently small, there are constants
$\mu_1\big(\delta,c_\delta,p\big) \geq 0$, $C_1=C_1(\delta,c_\delta,p)>0$ and
$C_2=C_2(\delta,c_\delta,p,T)>0$ such that
for any $\mu \geq \mu_1\big(\delta,c_\delta,p\big)$,  
for every $f \in L^{2p}$
there exists a function $u \in L^\infty(J_T, L^2(\Omega,L^2_\rho))$ 
for which the following are true.

{\rm(\textit{i})}
\begin{equation}
\label{grad_reg}
\sup_{t \in J_T}\|\mathbb E u^2(t)\|_p \leq \|f\|^2_{2p}, \quad 
\int_{J_T}\|\nabla v_p\|_2^2ds \leq C_1\|f\|_{2p}^p,
\end{equation}
\begin{equation}
\label{grad_reg_}
\mathbb E\bigl\langle \rho \big|\nabla \int_{J_T}u ds\big|^2 \bigr\rangle \leq 
C_2\|f\|_{2p}^2,
\end{equation}
where $v:=\mathbb E u^2$ and $v_p:=v|v|^{\frac{p}{2}-1}$, so, in particular, for a.e.~$\omega \in \Omega$, $\nabla \int_0^T u(s,\cdot,\omega)ds \in L_{\loc}^2(\mathbb R^d, \mathbb R^d)$ and hence $$b \cdot \nabla \int_{J_T} u(s,\cdot,\omega)ds \in L_{\loc}^1,$$ and, for every test function 
$\varphi \in C_c^\infty$, 
we have a.s.\,for all $t \in J_T$,
\begin{align}
& \langle u(t),\varphi \rangle - \langle f,\varphi\rangle  \notag \\
& + \mu\langle \int_0^t u ds,\varphi \rangle + \bigl\langle b \cdot \nabla \int_0^t  uds,\varphi \bigr\rangle - \sigma \bigl\langle \int_0^t u dB_s,\nabla \varphi \bigr\rangle + \frac{\sigma^2}{2}\bigl\langle \nabla \int_0^t u ds, \nabla \varphi\bigr\rangle=0.   \label{eq__}
\end{align}

{\rm(\textit{ii})} For any sequence of smooth vector fields $b_m \in C_c^\infty(\mathbb R^d,\mathbb R^d)$, $m=1,2,\dots,$ that are uniformly form-bounded in the sense that $b_m \in \mathbf{F}_\delta$ with $c_\delta$ independent of $m$, and are such that 
$$b_m \rightarrow b \text{ in $L^2_{\loc}(\mathbb R^d,\mathbb R^d)$ as $m \rightarrow \infty$},$$
we have for initial functions $f \in C_c^\infty$,
\begin{equation*}
u_m(t) \rightarrow u(t) \quad \text{ in } L^2(\Omega,L_\rho^2) \quad \text{ uniformly in $t \in J_T$},
\end{equation*}
where $u_m$ is the unique strong solution to \eqref{eq2} {\rm(}with $b=b_m${\rm)}.
\end{theorem}

An example of such smooth approximating vector fields $\{b_m\}$ is given in the next section.

The next theorem establishes the Sobolev regularity of $u$  
up to the initial time $t=0$.

\begin{theorem}
\label{thm1_reg}
Assume that $d \geq 3$, $b \in \mathbf{F}_\delta$ with 
$\sqrt{\delta}<\frac{\sigma^2}{2\beta_2}$ and $f \in W^{1,4}$.  
Let $\kappa$ be sufficiently small and $\mu_1(\delta, c_\delta, 2)$ be the constant in Theorem \ref{thm1} with $p=2$. For $\mu\ge 
\mu_1(\delta, c_\delta, 2)$, let $u$ be the process constructed in Theorem \ref{thm1}. There exists $\mu_2(\delta, c_\delta)\ge \mu_1(\delta, c_\delta, 2)$
such that for $\mu\ge \mu_2(\delta, c_\delta)$,
the following are true.

{\rm(a)} $\mathbb Eu^2$, $\mathbb E|\nabla u|^2 \in
L^\infty(J_T,L^2)$, so $u \in L^\infty(J_T, 
%L^2(\Omega,W^{1,2}_\rho))$, 
L^2(\Omega,W^{1,2}_\rho))$;

{\rm(b)} 
for any test function $\varphi \in C_c^\infty$, 
the process
$
t \mapsto \langle u(t), \varphi\rangle
$
is $(\mathcal F_t)$-progressively measurable and has a continuous $(\mathcal F_t)$-semi-martingale modification that satisfies a.s.\,for every $t \in J_T$,
\begin{align}
& \langle u(t),\varphi \rangle - \langle f,\varphi\rangle \notag    \\
& + \mu\int_0^t \langle u,\varphi \rangle ds + \int_0^t \bigl \langle b \cdot \nabla  u,\varphi \bigr\rangle ds  - \sigma  \int_0^t \langle u,\nabla \varphi \rangle dB_s + \frac{\sigma^2}{2}\int_0^t \bigl\langle  u, \Delta \varphi\bigr\rangle ds =0 \label{eq___}.
\end{align}

Moreover, if $\sqrt{\delta}<\frac{\sigma^2}{2\beta_{2q}}$ for some  $q=1,2,\dots$, 
then there exists constants $\mu_2(\delta, c_\delta, q)\ge \mu_1(\delta, c_\delta, 2q)$ (with $\mu_2(\delta, c_\delta, 1)$ equal to the $\mu_2(\delta, c_\delta)$ above) and $C_1=C_1(\delta, c_\delta, q)>0$ such that when $\mu\ge \mu_2(\delta, c_\delta, q)$ and $f \in W^{1,4q}$, we have
\begin{equation}
\label{grad_reg2}
\sup_{0 \leq \alpha \leq 1}\bigl\|\mathbb E|\nabla u|^{2q}\bigr\|_{L^{\frac{2}{1-\alpha}}(J_T,L^{\frac{2d}{d-2+2\alpha}})} \leq 
C_1\|\nabla f\|^{2q}_{4q}.
\end{equation}
In particular,  there exists $C_2>0$ such that
\begin{equation}
\label{grad_reg3}
\sup_{t \in J_T}\mathbb E\langle \rho |\nabla u|^{2q}\rangle \leq  C_2\|\nabla f\|^{2q}_{4q}.
\end{equation}
If $2q>d$, then for a.e. $\omega \in \Omega$, $t \in J_T$, the function $x \mapsto u(t,x,\omega)$ is H\"{o}lder continuous, possibly after 
modification on a set of measure zero in $\mathbb R^d$ (in general, depending on $\omega$).
\end{theorem}

\begin{theorem}
\label{thm2}
Assume that $d \geq 3$, $b \in \mathbf{F}_\delta$ with 
$\sqrt{\delta}<\frac{\sigma^2}{2\beta_2}$ and $f\in W^{1, 4}$. 
Provided $\kappa$ is sufficiently small, there exists $\mu_3=\mu_3(\delta, c_\delta)\ge 0$ such that for $\mu\ge \mu_3(\delta, c_\delta)$, 
\eqref{eq2}
has a unique solution in the class of functions satisfying 
{\rm(a)}, {\rm(b)} of Theorem \ref{thm1_reg}.
\end{theorem}

A function satisfying (a), (b) of Theorem \ref{thm1_reg}
will be called a weak solution of \eqref{eq2}. This definition of weak solution is  close to \cite[Definition 2.13]{BFGM}. It should be noted however that the authors in \cite{BFGM} prove their uniqueness result, in the time-dependent case, in a larger class of weak solutions (not requiring any differentiability, see \cite[Definition 3.3]{BFGM}) but under additional assumptions on $b$. 
Specialized to the time-dependent case, they assume that $b$ satisfies
\begin{equation}
\label{div}
{\rm div\,}b \in L^d + L^\infty 
\end{equation}
in addition to $|b| \in L^d + L^\infty$.
The latter is needed to establish \eqref{grad_reg2} for solutions of the adjoint equation to the STE, i.e.\,the stochastic continuity equation (which allows to prove an even stronger 
%result, i.e.\,the uniqueness of 
%RS I do not like two i.e.'s in the same sentence
result: the uniqueness of 
weak solution to the corresponding  random transport equation),
see \cite[Sect.\,3]{BFGM}. 

We expect that an analogue of \eqref{div} for $b \in \mathbf{F}_\delta$ can be found with some additional effort. However, 
we will not address this matter in this paper. 
Of course, in the case $b \in \mathbf{F}_\delta$, ${\rm div\,}b=0$, one has \eqref{grad_reg2} for solutions to the stochastic continuity equation, so one can prove the uniqueness for \eqref{eq2} by repeating the argument in \cite[Sect.\,3]{BFGM}. 

The proof of the uniqueness result in Theorem \ref{thm2} (see Section \ref{unique_sect}) adopts the method of \cite[Sect.\,3]{BFGM}.

\begin{remark}[On applications to SDEs]
\label{rem_sdes}
Armed with Theorems \ref{thm1} and \ref{thm1_reg}, one can repeat the argument in \cite[Sect.\,4]{BFGM} to prove the following result.
Assuming that $b \in \mathbf{F}_\delta$ with $\delta$ sufficiently small,  
there exists a stochastic Lagrangian flow for SDE \eqref{sde}, i.e.\,a measurable map $\Phi:J_T \times \mathbb R^d \times \Omega \rightarrow \mathbb R^d$ such that, for a.e.~$x \in \mathbb R^d$, the process $t \mapsto \Phi_t(x,\omega)$ is a strong solution of the SDE \eqref{sde}:
\begin{equation}
\label{sde3}
\Phi_t(x,\omega)=x-\int_0^t b(s,\Phi_r(x,\omega))dr+\sigma B_t(\omega), \quad \text{a.s.}, \quad t \in J_T,
\end{equation}
and $\Phi_t(x,\cdot)$ is $\mathcal F_t$-progressively measurable.
If also $\sqrt{\delta}<\frac{\sigma^2}{2\beta_{2q}}$, $q=1,2,\dots$, then $\Phi_t(\cdot,\omega) \in W_{\loc}^{1,2q}$ ($t \in J_T$) for a.e. $\omega \in \Omega$. Moreover, $\Phi_t$ is unique, i.e.\,any two such stochastic flows coincide a.s. for every $t>0$ for a.e.\,$x$. 
\end{remark}

\begin{remark}[STE with time-dependent $b$]
\label{time_rem}
The proof of the key result of this paper (Proposition \ref{apr_prop-2} below, i.e.\,a priori Sobolev regularity of solutions of the STE) carries over, without change, to the time-dependent form-bounded vector fields: 

\begin{definition}
A vector field $b \in L^2_{\loc}\bigl([0,\infty) \times \mathbb R^d,\mathbb R^d\bigr)$  is said to be form-bounded with relative bound  
$\delta>0$, written as $b\in \widetilde{\mathbf{F}}_\delta$,
if $|b| \in L^2_{\loc}([0,\infty) \times \mathbb R^d)$ and
$$
\int_0^\infty \|b(t,\cdot)\phi(t,\cdot)\|_2^2 dt \leqslant \delta \int_0^\infty\|\nabla \phi(t,\cdot)\|_2^2 dt+\int_0^\infty g(t)\|\phi(t,\cdot)\|_2^2dt
$$
for some $g=g_\delta \in L^1_{\loc}[0,\infty)$, for all $\phi \in C_c^\infty([0,\infty)  \times \mathbb R^d)$.
\end{definition}

The class $\widetilde{\mathbf{F}}_\delta$ contains  vector fields
$$
|b(\cdot,\cdot)| \in L^q\bigl([0,\infty),L^r+L^\infty\bigr), \quad \frac{d}{r}+\frac{2}{q} \leqslant 1,
$$
with $\delta$ that can be chosen arbitrarily small 
(using H\"{o}lder's inequality and the Sobolev embedding theorem).
Another example is
$$
|b(t,x)|^2 \leqslant c_1|x-x_0|^{-2} + c_2|t-t_0|^{-1}\bigl(\log(e+|t-t_0|^{-1}) \bigr)^{-1-\varepsilon}, \quad \varepsilon>0, \quad (t,x) \in [0,\infty) \times \mathbb R^d,
$$
which belongs to the class  
$\widetilde{\mathbf{F}}_\delta$ with 
$\delta=c_1\left(2/(d-2) \right)^2$  (using Hardy's inequality). 

We plan to address the regularity theory of the STE with 
$b\in \widetilde{\mathbf{F}}_\delta$ elsewhere.
\end{remark}

\section{A priori estimates} 
\label{apr_sect}

Assume $b \in \mathbf{F}_\delta$.
In the remainder of this paper, we fix some  
$b_m \in C_c^\infty(\mathbb R^d, \mathbb R^d)$ such that 
$$b_m \rightarrow b \text{ in $L^2_{\loc}(\mathbb R^d,\mathbb R^d)$ as $m \rightarrow \infty$}$$
and for every $m=1,2,\dots$
$$
\| b_m\varphi\|_2^2 \leq \delta \| \nabla\varphi\|_2^2 + c_\delta \|\varphi\|_2^2, \quad \varphi \in W^{1,2} 
$$
with $c_\delta$ independent of $m$ (see example of such $b_m$ below).
Let $f \in C_c^\infty$. 
Let $u_m$ be the unique 
strong solution to
\begin{equation}
\label{eq5}
u_m(t) - f+ \mu \int_0^t u_m ds +\int_0^t b_m \cdot \nabla uds + \sigma\int_0^t \nabla u_m \circ dB_s=0 \text{ a.s.},
 \quad t \in J_T=[0,T].
\end{equation}
Then, by \cite[Section 6.1]{Ku}, 
for any $p, r\ge 1$ and any multiindex $\alpha=(\alpha_1, \dots, \alpha_d)$ of
non-negative integers,
$$
\mathbb E\left(\left|
D^\alpha u_m
\right|^p\right) \in L^\infty(J_T \times \mathbb R^d)
$$ 
and 
$$
\int_{\mathbb R^d}(1+|x|^r)\bigl( \mathbb E|u_m|^p + \mathbb E|\nabla u_m|^p\bigr) dx \in L^\infty(J_T).
$$

\begin{remark}[Example of $\{b_m\}$]
Denote by $\mathbf{1}_m$ the indicator of 
$\{|x| \leq m, |b(x)| \leq m\}$, and by $\eta_m \in C_c^\infty$ a $[0,1]$-valued function such that $\eta_m = 1 $ on $B(0,m)$. Consider
\begin{equation}
\label{b_m_def}
\tag{$\ast$}
b_m:=\eta_m e^{\epsilon_m\Delta}(\mathbf{1}_m b),
\end{equation}
where $\epsilon_m \downarrow 0$ is to be chosen. 

First, let us show  that, for any $\{\gamma_m\} \downarrow 0$ we can select $\{\epsilon_m\} \downarrow 0$ in the definition of $b_m$ so that 
$$
b_m \in \mathbf{F}_{\delta_m} \quad \text{ with $\delta_m=(\sqrt{\delta}+\sqrt{\gamma_m})^2 \downarrow \delta$ and $c_{\delta_m} \leq 2c_\delta$  
starting from some $m$ on}. 
$$
Since $b \in \mathbf{F}_\delta$, there exists $\lambda\ge 0$ such that $\||b|(\lambda-\Delta)^{-\frac{1}{2}}\|_{2 \rightarrow 2} \leq \sqrt{\delta}$. Then $c_\delta=\lambda\delta$.
We claim that, we can select $\{\epsilon_m\} \downarrow 0$ fast enough so that
\begin{equation}
\label{b_m_cond}
\tag{$\ast\ast$}
\||b_m|(\lambda-\Delta)^{-\frac{1}{2}}\|_{2 \rightarrow 2} 
\leq \sqrt{\delta_m}.
\end{equation}
Once this claim is proven, we will have $c_{\delta_m}=\lambda \delta_m \leq 2c_\delta$ starting from some $m$ on, which implies the required. Now we prove the claim.
We have $$b_m=\mathbf{1}_m b + (b_m-\mathbf{1}_m b),$$ where, 
clearly, $\||\mathbf{1}_mb|(\lambda-\Delta)^{-\frac{1}{2}}\|_{2 \rightarrow 2} \leq \sqrt{\delta}$ for every $m$, while 
$b_m-\mathbf{1}_m b \in L^d$. 
It follows from H\"{o}lder's inequality and the Sobolev embedding theorem that
for any $g \in L^2$,
\begin{align*}
\||b_m-\mathbf{1}_m b|(\lambda-\Delta)^{-\frac{1}{2}}g\|_2 & \leq \|b_m-\mathbf{1}_m b\|_d \|(\lambda-\Delta)^{-\frac{1}{2}}g\|_{\frac{2d}{d-2}} \leq c\|b_m-\mathbf{1}_m b\|_d \|g\|_2.
\end{align*}
It is easily seen that, for every $m$,
the norm $\|b_m-\mathbf{1}_m b\|_d$ can be made smaller than  $c^{-1}\gamma_m$  by selecting $\{\epsilon_m\} \downarrow 0$ sufficiently rapidly.
Thus
$$
\|(b_m-\mathbf{1}_m b)(\lambda-\Delta)^{-\frac{1}{2}}\|_{2 \rightarrow 2} \leq \gamma_m.
$$
Now \eqref{b_m_cond} follows.

Finally, to have $b_m$ form-bounded with the original relative bound $\delta$, it suffices to multiply $b_m$ in \eqref{b_m_def} by $\frac{\delta}{\delta_m}$. (Although, to carry 
out the proofs of Theorems \ref{thm1}-\ref{thm2},
the last step is not necessary since all our assumptions on $\delta$ are strict inequalities.)
\end{remark}

We prove the next proposition under 
more general assumptions on $\delta$ and $p$ than in Theorem \ref{thm1}.

\begin{proposition}
\label{apr_prop}
Let $b \in \mathbf{F}_\delta$ with $\sqrt{\delta}<\sigma^2$. Let $T>0$, $p \in (p_c,\infty)$, $p_c:=\big(1-\frac{\sqrt{\delta}}{\sigma^2}\big)^{-1}$. 
Let $f \in C_c^\infty$, let $b_m$ and $u_m$ be as above.
There exist constants 
$\mu(\delta,c_\delta,p) \geq 0$, $C_1=C_1(\delta,c_\delta,p)>0$ and
$C_2=C_2(\delta,c_\delta,p,T)>0$ independent of $m$
such that  for any 
$\mu\ge \mu\big(\delta, c_\delta,p\big)$ and $m=1, 2,\dots$, the following are true:

{\rm(\textit{i})} 
\begin{equation}
\label{e1}
\tag{$E_1$}
\sup_{t \in J_T}\|\mathbb Eu_m^2(t)\|_p \leq \|f\|^2_{2p}, \quad \int_{J_T}\|\nabla v_p\|_2^2ds \leq 
C_1\|f\|_{2p}^p,
\end{equation}
where $v:=\mathbb E u^2$ and $v_p:=v|v|^{\frac{p}{2}-1}$;

{\rm(\textit{ii})} if $\sqrt{\delta}<\frac{\sigma^2}{2}$, then
\begin{equation}
\label{e2}
\tag{$E_2$}
 \mathbb E\bigl\langle \rho \left(\nabla \int_{J_T} u_m(s) ds\right)^2 \bigr\rangle 
 \leq C_2\|f\|_{2p}^2.
\end{equation}
\end{proposition}

\begin{proposition}
\label{apr_prop-2}
%Let $b \in \mathbf{F}_\delta$. Let 
%$f \in C_c^\infty$, let $b_m$ and $u_m$ be as 
%above.
Let $b \in \mathbf{F}_\delta$ and $f \in C_c^\infty$, let $b_m$ and $u_m$ be as 
above.
For every $q \geq 1$,  there exists constants 
$\mu(\delta,c_\delta,q) \geq 0$  and $C=C(\delta,c_\delta,q)>0$
independent of $m$
such that if 
$\sqrt{\delta}<\frac{\sigma^2}{2\beta_{2q}}$ and $\mu\ge 
\mu(\delta, c_\delta, q)$,
then 
\begin{equation}
\label{apr_est}
\tag{$E_3$}
\sup_{0 \leq \alpha \leq 1}\bigl\|\mathbb E|\nabla u_m|^{2q}\bigr\|_{L^{\frac{2}{1-\alpha}}\bigl([0, T],L^{\frac{2d}{d-2+2\alpha}}\bigr)} \leq 
C\|\nabla f\|^{2q}_{4q}.
\end{equation}
\end{proposition}

\begin{proof}[Proof of Proposition \ref{apr_prop}]
For brevity, we write $u$ for $u_m$ in this proof.
The identity \eqref{strat_id} allows us to rewrite \eqref{eq5} as 
\begin{equation}
\label{eq5_}
u(t,\cdot) - f + \mu\int_0^t u ds +\int_0^t b_m \cdot \nabla u ds + \sigma\int_0^t \nabla u dB_s -\frac{\sigma^2}{2}\int_0^t \Delta u ds=0 \quad \text{a.s.,}\quad t \in J_T.
\end{equation}
Below we will be appealing to \eqref{eq5_}.

\medskip
 
We first prove \eqref{e1}.
Applying It\^{o}'s formula to $u^2$, we obtain, in view of \eqref{eq5_}, 
\begin{align*}
u^2(t)-f^2=-2\mu\int_0^t u^2 ds 
-\int_0^t b_m\cdot\nabla u^2 ds 
-\sigma\int_0^t \nabla u^2 dB_s + \frac{\sigma^2}{2}\int_0^t \Delta u^2 ds.
\end{align*}
Since $t \mapsto \int_0^t \nabla u^2  dB_s$ is a martingale, 
$v=\mathbb E u^2$ satisfies
$$
\partial_t v = - 2\mu v 
-b_m\cdot\nabla v + \frac{\sigma^2}{2} \Delta v, \quad v(0)=f^2.
$$
We multiply the last equation by $v|v|^{p-2}$ and integrate by parts 
(recall that $v_p=v|v|^{\frac{p}{2}-1}$),
$$
\frac{1}{p}\partial_t \langle |v_p|^2\rangle + 2\mu \langle |v_p|^2\rangle
+ \frac{4}{pp'}\frac{\sigma^2}{2} \langle |\nabla v_p|^2\rangle - \frac{2}{p}\langle b_m \cdot \nabla v_p,v_p\rangle \leq 0,
$$
so applying the quadratic inequality we have (for $\varepsilon>0$)
$$
\partial_t \langle |v|^p\rangle + 2p \mu \langle |v|^p\rangle
+ \frac{2\sigma^2}{p'}\langle |\nabla v_p|^2\rangle - 2\biggl(\varepsilon \langle |\nabla v_p|^2\rangle + \frac{1}{4\varepsilon} \langle b_m^2 v_p^2 \rangle  \biggr) \leq 0.
$$
Finally, by our assumption on $b_m$,
$$
\partial_t \langle |v|^p\rangle + 2p \mu \langle |v|^p\rangle 
+ \frac{2\sigma^2}{p'}\langle |\nabla v_p|^2\rangle - 2\biggl(\varepsilon \langle |\nabla v_p|^2\rangle + \frac{\delta}{4\varepsilon} \langle |\nabla v_p|^2\rangle + \frac{c_\delta}{4\varepsilon}\langle |v|^p\rangle  \biggr) \leq 0.
$$
Taking $\varepsilon=\frac{\sqrt{\delta}}{2}$ in the last inequality and integrating with respect to $t$, we obtain for $t>0$
$$
\langle |v(t)|^p\rangle + 2\biggl(\frac{\sigma^2}{p'}-\sqrt{\delta} \biggr)\int_0^t \langle |\nabla v_p|^2\rangle ds + 
\biggl[ 2p\mu - \frac{c_\delta}{2\sqrt{\delta}}\biggr]
\int_0^t \langle |v|^p\rangle ds  \leq \|f^2\|^p_{p},
$$
where $\frac{\sigma^2}{p'}-\sqrt{\delta}>0$ since $p>p_c$. 
Taking $\mu \geq \frac{c_\delta}{4\sqrt{\delta}p}$, we arrive at \eqref{e1}.

\medskip

Now we deal with \eqref{e2}. 
Let $\mu \geq \frac{c_\delta}{4\sqrt{\delta}p}$ as above.
By \eqref{e1},  
\begin{equation}
\label{eee}
\sup_{t \in J_T}\bigl\langle \rho \mathbb E u^2(t)\bigr\rangle \leq \|\rho\|_{p'}\sup_{t \in J_T}\|\mathbb Eu^2(t)\|_p \leq  
c_1\|f\|_{2p}^2,
\end{equation}
since $\theta>\frac{d}{2}$  in the definition of $\rho$.

We multiply \eqref{eq5_} by $\rho \int_0^t u ds$, integrate, and take expectation, to get
\begin{align}
\label{chain}
\mathbb E\bigl\langle \rho\int_0^t u ds,u(t)\bigr\rangle & = \mathbb E\bigl\langle \rho\int_0^t u ds,f\bigr\rangle - \mathbb E\bigl\langle \rho\int_0^t u ds,b_m \cdot \nabla  \int_0^t u ds\bigr\rangle  \\
&  - \sigma \mathbb E\bigl\langle \rho\int_0^t u ds,\int_0^t \nabla u dB_s\bigr\rangle + \frac{\sigma^2}{2}\mathbb E
\bigl\langle \rho\int_0^t u ds,\int_0^t \Delta u ds\bigr\rangle + \mu\mathbb E \bigl\langle \rho \int_0^t u ds,\int_0^t u ds\bigr\rangle \notag \\
& =:I_1+I_2+I_3+I_4 + I_5. \notag
\end{align}
Denote the left-hand side  of \eqref{chain} by $I_0$.
Set
$$
U:=\int_0^t u ds.
$$
By H\"{o}lder's inequality and \eqref{eee}, 
\begin{align}
\label{e4}
\mathbb E \bigl\langle \rho U^2\bigr\rangle \leq
t \bigl\langle \rho  \int_0^t \mathbb E u^2 ds \bigr\rangle \leq 
t^2 c_1\|f\|^2_{2p}.
\end{align}

Integrating by parts in $I_4$ and using the quadratic inequality, we have
\begin{align*}
\frac{2}{\sigma^2}I_4 &=-E\bigl\langle \rho |\nabla U|^2\bigr\rangle 
- E\bigl\langle U\nabla \rho,\nabla U \bigr\rangle \\
& \leq 
-E\bigl\langle \rho |\nabla U|^2\bigr\rangle +\alpha E \bigl\langle |\nabla \rho|U^2\bigr\rangle + \frac{1}{4\alpha} E \bigl\langle |\nabla \rho| |\nabla U|^2\bigr\rangle \qquad (\alpha>0)\\
& \text{ (we are applying \eqref{two_est} in the last term, and \eqref{two_est}, \eqref{e4} in the middle term)} \\
& \leq 
-\left(1-\frac{\theta\sqrt{\kappa}}{4\alpha} \right)E\bigl\langle \rho |\nabla U|^2\bigr\rangle +\theta\sqrt{\kappa}\alpha T^2 
c_1\|f\|_{2p}^2.
 \end{align*}
Substituting the last estimate into \eqref{chain}, we obtain
\begin{align}
\label{e5}
\frac{\sigma^2}{2}\left(1-\frac{\theta\sqrt{\kappa}}{4\alpha} \right) \mathbb E\bigl\langle \rho |\nabla U|^2\bigr\rangle \leq \frac{\sigma^2}{2} \theta\sqrt{\kappa}\alpha T^2 
c_1\|f\|_{2p}^2 + |I_0|+|I_1|+|I_2|+|I_3|+|I_5|.
\end{align}

We now estimate $|I_i|$, $i=0,1,2,3,5$. 
By \eqref{eee} and \eqref{e4},
$$
|I_0| \leq \left(\mathbb E\bigl\langle \rho U^2\bigr\rangle \right)^{\frac{1}{2}} 
\left(\mathbb E\bigl\langle \rho u^2(t)\bigr\rangle \right)^{\frac{1}{2}} \leq 
c_2\|f\|_{2p}^2.
$$

Similarly, 
$$
|I_1| \leq c_3\|f\|_{2p}^2, \quad |I_5| \leq \mu c_4\|f\|_{2p}^2.
$$

Next, applying the quadratic inequality, we get
\begin{align*}
|I_2| 
& \leq \nu \mathbb E \bigl\langle \rho b_m^2 U^2 \bigr\rangle + \frac{1}{4\nu} \mathbb E \bigl\langle \rho|\nabla U|^2 \bigr\rangle \qquad (\nu>0)\\
& \text{\big(in the first term, we apply $b_m \in \mathbf{F}_\delta$ with $\varphi:=\sqrt{\rho }U$)} \\
& \leq \nu\bigl(\delta \mathbb E\langle |\nabla (\sqrt{\rho}U)|^2\rangle + c_\delta \mathbb E\langle \rho U^2 \rangle  \bigr) +   \frac{1}{4\nu} \mathbb E 
\bigl\langle \rho|\nabla U|^2 \bigr\rangle \\
& (\text{in the first term, we use $(a+c)^2 \leq (1+\epsilon) a^2+(1+\frac{1}{\epsilon})c^2$, $\epsilon>0$}) \\
& \leq \nu \delta (1+\epsilon)\mathbb E \bigl\langle  \left|\sqrt{\rho}\,\nabla U\right|^2 \bigr\rangle
 + \nu \delta\big(1+\frac{1}{\epsilon}\big) \mathbb E \bigl\langle \left|U\nabla \sqrt{\rho}\right|^2 \bigr\rangle
 + \nu c_\delta \mathbb E\langle \rho U^2 \rangle + \frac{1}{4\nu} \mathbb E \bigl\langle \rho \left|\nabla  U \right|^2 \bigr\rangle \\
& \text{(in the second term, we apply \eqref{two_est} and then use \eqref{e4}; also, we apply \eqref{e4} in the last term)} \\
& \leq \left( \nu \delta(1+\epsilon) + \frac{1}{4\nu}\right) \bigl\langle \rho \left|\nabla U \right|^2 \bigr\rangle +  
T^2 c_5\|f\|_{2p}^2, \quad 
c_5=c_5(\nu,\delta,c_\delta,\theta,\kappa,\epsilon).
\end{align*}
In the current setting, we have $\int_0^t \nabla u dB_s=\nabla \int_0^t u dB_s$
(see, for instance, \cite{HN}). Thus,  integrating by parts,
we obtain 
$$
I_3= \sigma \mathbb E\bigl\langle \rho \nabla U,\int_0^t u dB_s\bigr\rangle +
\sigma \mathbb E\bigl\langle U\nabla \rho,\int_0^t u dB_s\bigr\rangle,
$$ 
so 
\begin{align*}
|I_3| 
& \leq \sigma \bigl(\mathbb E\bigl\langle 
\rho \left|\nabla U \right|^2 \bigr\rangle\bigr)^{\frac{1}{2}} \bigl(\mathbb E\bigl\langle 
\rho \left(\int_0^t u dB_s \right)^2 \bigr\rangle\bigr)^{\frac{1}{2}} \\
& + 
\sigma \left(\mathbb E\bigl\langle |\nabla \rho |\, U^2 \bigr\rangle\right)^{\frac{1}{2}}
\bigl(\mathbb E\bigl\langle |\nabla \rho | \bigl(\int_0^t u dB_s \bigr)^2 \bigr\rangle\bigr)^{\frac{1}{2}}
\\
& \text{ (we use \eqref{two_est} and apply the It\^{o} isometry)} \\
& \leq \sigma 
\left(\mathbb E\bigl\langle \rho \left|\nabla U \right|^2 \bigr\rangle\right)^{\frac{1}{2}}
\bigl(\mathbb E\bigl\langle \rho \int_0^t u^2 ds \bigr\rangle\bigr)^{\frac{1}{2}} 
 \\
&+ 
\theta\sqrt{\kappa}\sigma \left(\mathbb E\bigl\langle \rho U^2 \bigr\rangle\right)^{\frac{1}{2}}
\left(\mathbb E\bigl\langle \rho \int_0^t u^2 ds \bigr\rangle\right)^{\frac{1}{2}}
\\
& \text{(we apply the quadratic inequality in the first term and then use \eqref{e4})} \\
& \leq 
\sigma \gamma \mathbb E\bigl\langle \rho \left|\nabla U \right|^2 \bigr\rangle
+ \frac{\sigma T^2c_1 } {4\gamma}\|f\|_{2p}^2 +
\theta\sqrt{\kappa}\sigma  T^2 c_1\|f\|_{2p}^2 \qquad (\gamma>0).
\end{align*}

Substituting the above estimates on $|I_0|$, $|I_1|$, $|I_2|$, $|I_3|$ and $|I_5|$
in \eqref{e5}, we obtain
$$
\left(\frac{\sigma^2}{2}- \nu \delta (1+\epsilon) - \frac{1}{4\nu} -\sigma\gamma-\frac{\sigma^2}{2}\frac{\theta\sqrt{\kappa}}{4\alpha} \right) \mathbb E\bigl\langle \rho \left|\nabla U \right|^2\bigr\rangle \leq 
c_6\|f\|_{2p}^2
$$
for an appropriate constant 
$c_6=c_6(\alpha,\gamma,\nu,\delta,\theta,\kappa,\epsilon,c_\delta,\mu)<\infty$. 
Take $\nu=(2\sqrt{\delta})^{-1}$. 
Since $\sqrt{\delta}< \frac{\sigma^2}{2}$ by assumption, we can
select $\gamma$, $\epsilon$ sufficiently small and $\alpha$ sufficiently large so that 
$$\frac{\sigma^2}{2}- \big(\nu \delta + \frac{1}{4\nu}\big) - \nu\delta\varepsilon - \sigma\gamma-\frac{\sigma^2}{2}\frac{\theta\sqrt{\kappa}}{4\alpha}>0,$$ 
and thus  \eqref{e2} follows with constant 
$C_2=c_6\big(\frac{\sigma^2}{2}- \nu \delta (1+\epsilon) - \frac{1}{4\nu} -\sigma\gamma-\frac{\sigma^2}{2}\frac{\theta\sqrt{\kappa}}{4\alpha} \big)^{-1}$.
\end{proof}

\begin{remark}
In Proposition \ref{apr_prop}, the interval $(p_c,\infty)$ of admissible values of $p$ decreases to the empty set as $\sqrt{\delta} \uparrow \sigma^2$. 
In fact, one can show that if $b \in \mathbf{F}_\delta$, $\sqrt{\delta}<\sigma^2$ and $b_m \in C_c^\infty$ are as above, then the limit
$$
s\mbox{-}L^p\mbox{-}\lim_m e^{-t\Lambda_m} \quad \text{(loc.\,uniformly in $t \geq 0$)}, \quad p>p_c,
$$ 
where
$
\Lambda_m=-\frac{\sigma^2}{2}\Delta + b_m \cdot \nabla$, $D(\Lambda_m)=W^{2,p}$,
exists and determines a $L^\infty$ contraction, quasi contraction holomorphic semigroup in $L^p$, say, $e^{-t\Lambda}$, 
see \cite[Theorems 4.2, 4.3]{KiS3}. The operator $\Lambda$ is an appropriate operator realization of the formal operator $-\frac{\sigma^2}{2}\Delta + b \cdot \nabla$ in $L^p$. One can compare this result with the example in \cite[Sect.\,7]{BFGM}, where the authors show that the SDE 
$$
X_t=-\int_0^t b(X_s) ds + \sigma B_t, \quad b(x)=\sqrt{\delta}\frac{d-2}{2}|x|^{-2}x \in \mathbf{F}_\delta,
$$
corresponding to operator $-\frac{\sigma^2}{2}\Delta + b \cdot \nabla$,
does not have a weak solution if $\sqrt{\delta}>\sigma^2$.  
\end{remark}

\medskip

\begin{proof}[Proof of Proposition \ref{apr_prop-2}]
For any multiindex $I$ with entries in $\{1, \dots, d\}$, i.e., an element of $\{1,\dots,d\} \times \dots \times \{1,\dots,d\}$, say, $p$ times, we write $|I|=p$. 
For any such multiindex $I$ and $l\in \{1,\dots,d\}$, we denote by 
$I-l$  the multiindex obtained from $I$ by dropping an index of value $l$. Let $I-l+k$ be the multiindex $I$ with an index of value $l$ dropped and replaced with an index of value $k$. 
It does not matter from which component the value $l$ is dropped.

For brevity, we write $u$ for $u_m$ in this proof.
Set $$w_r:=\partial_{x_r} u, \quad 1 \leq r \leq d,
$$
where $u$ is the strong solution of \eqref{eq5}, and 
$$
w_I:=\prod_{r \in I} \partial_{x_r} u.
$$

\smallskip

\textit{Step 1.}~We apply It\^{o}'s formula in Stratonovich form to $w_I$, obtaining
$$
w_I(t)-\prod_{r \in I}\partial_{x_r} f=\sum_{r \in I} \int_0^t w_{I-r}(s) \circ dw_r(s).
$$
Next, differentiating \eqref{eq5_} in $x_r$ and then substituting the resulting expression for $dw_r$ into the previous formula, we obtain
$$
w_I(t)-\prod_{r \in I}\partial_{x_r} f = -\mu\int_0^t w_I ds - \sum_{r \in I} \int_0^t w_{I-r}\bigl(b_m\cdot\nabla w_r + \partial_{x_r}b \cdot\nabla u\bigr)ds - \sigma\sum_{r \in I}\int_0^t w_{I - r} \nabla w_r \circ dB_s.
$$
Let $b^k_m$, $k=1,\dots,d$, be the components of the vector field $b_m$. We have
\begin{align*}
w_I(t)& -\prod_{r \in I}\partial_{x_r} f=-\mu\int_0^t w_I ds -\sum_{r \in I} \int_0^t w_{I-r}\bigl(b_m\cdot\nabla w_r + \partial_{x_r}b_m \cdot\nabla u\bigr)ds - \sigma\int_0^t \nabla w_I \circ dB_s \\
&\text{(we use $\int_0^t \nabla w_I \circ dB_s=\int_0^t \nabla w_I dB_s-\frac{1}{2}\sum_{k=1}^d[\partial_{x_k}w_I,B^k]_t$)} \\
&=-\mu\int_0^t w_I ds -\sum_{r \in I} \int_0^t w_{I-r}\bigl(b_m\cdot\nabla w_r + \partial_{x_r}b_m  \cdot\nabla u\bigr)ds - \sigma\int_0^t \nabla w_I  dB_s + \frac{\sigma^2}{2} \int_0^t \Delta w_I ds \\
&=-\mu\int_0^t w_I ds-\int_0^t b_m\cdot\nabla w_I ds -\sum_{r \in I} \sum_{k=1}^d \int_0^t  \partial_{x_r}b^k_m w_{I-r+k} ds -  \sigma\int_0^t \nabla w_I  dB_s + \frac{\sigma^2}{2} \int_0^t \Delta w_I ds.
\end{align*}
Put
$$
v_I:=\mathbb E[w_I].
$$
Since $t \mapsto \int_0^t \nabla w_I  dB_s$ is a martingale, 
$v_I$ satisfies
$$
v_I(t) - \prod_{r \in I}\partial_{x_r} f=-\mu\int_0^t v_I ds -\int_0^t b_m \cdot \nabla v_I ds - \sum_{r \in I} \sum_{k=1}^d \int_0^t \partial_{x_r}b^k_m v_{I-r+k}ds + \frac{\sigma^2}{2} \int_0^t \Delta v_I ds,
$$
i.e.,
\begin{equation}
\label{cp}
\partial_t v_I=- \mu v_I + \frac{\sigma^2}{2}\Delta v_I - b_m\cdot \nabla v_I - \sum_{r \in I} \sum_{k=1}^d \partial_{x_r}b^k_m v_{I-r+k}, \quad v_I(0)=\prod_{r \in I}\partial_{x_r} f.
\end{equation}

\textit{Step 2.}~We multiply the equation in \eqref{cp} by  $v_I$, and integrate:
\begin{align*}
\frac{1}{2}\partial_t \bigl\langle v_I^2 \bigr\rangle + \mu \langle v_I^2 \rangle + \frac{\sigma^2}{2} \bigl\langle (\nabla v_I)^2 \bigr\rangle  = - \bigl\langle v_I,b_m\cdot\nabla v_I\bigr\rangle - \bigl\langle v_I, \sum_{r \in I} \sum_{k=1}^d \partial_{x_r}b^k_m  v_{I-r+k}\bigr\rangle.
\end{align*}
Then, for every $t \in J_T$,
\begin{align}
\label{wI_id}
\frac{1}{2}\bigl\langle v_I^2(t) \bigr\rangle & - \frac{1}{2} \bigl\langle v_I^2(0) \bigr\rangle +\mu\int_0^t v_I^2 ds  + \frac{\sigma^2}{2}  \int_0^ t \bigl\langle (\nabla v_I)^2 \bigr\rangle ds \\ 
\notag
& = - \int_0^t\bigl\langle v_I,b_m\cdot\nabla v_I\bigr\rangle ds - \int_0^t \bigl\langle v_I, \sum_{r \in I} \sum_{k=1}^d \partial_{x_r}b^k_m v_{I-r+k}\bigr\rangle ds =:-S^1_I-S^2_I. \notag
\end{align}

We estimate $|S^1_I|$ and $|S^2_I|$ as follows:
\begin{align}
|S^1_I|  & \leq \biggl|\int_0^t \bigl\langle v_I,b_m\cdot\nabla v_I\bigr\rangle ds\biggr| \leq \gamma \int_0^t \bigl\langle (\nabla v_I)^2\bigr\rangle ds + \frac{1}{4\gamma} \int_0^t\bigl\langle v_I^2 b^2_m\bigr\rangle ds 
\nonumber\\
& \text{(we use $b_m \in \mathbf{F}_\delta$))}
\nonumber\\
& \leq 
\biggl(\gamma + \frac{\delta}{4\gamma}\biggr) \int_0^t \bigl\langle (\nabla v_I)^2\bigr\rangle ds  + \frac{c_\delta}{4\gamma}\int_0^t\langle v_I^2\rangle.
\label{e:si1}
\end{align}
Next, integrating by parts, and applying the quadratic inequality, we have
\begin{align*}
|S^2_I|&=\biggl|-\int_0^t\sum_{r \in I} \sum_{k=1}^d \langle (v_{I-r+k}\partial_{x_r}v_I + v_I\partial_{x_r}v_{I-r+k})b^k_m \rangle\biggr|ds \\
& \leq \alpha \int_0^t\sum_{r \in I} \sum_{k=1}^d \bigl\langle (\partial_{x_r}v_I)^2 + (\partial_{x_r}v_{I-r+k})^2 \bigr\rangle ds + \frac{1}{4\alpha} \int_0^t \sum_{r \in I} \sum_{k=1}^d \bigl\langle v_{I-r+k}^2 (b^k_m)^2 + v_I^2 (b^k_m)^2 \bigr\rangle ds.
\end{align*}

Let $q=1,2,\dots$. 
Summing over all $I$ with $|I|=2q$ and noticing that every multiindex of length $2q$
is counted $4qd$ times, we obtain
\begin{align*}
\sum_I |S^2_I| & \leq 
4\alpha qd
\sum_I \int_0^t \bigl\langle |\nabla v_I|^2\bigr\rangle ds + 
\frac{qd}{\alpha} 
\sum_I \int_0^ t\bigl\langle v_I^2 b^2_m\bigr\rangle ds 
\\
& \text{(use $b_m \in \mathbf{F}_\delta$ in the second term)}
\\ 
& \leq 
4\alpha qd
\sum_I \int_0^t \bigl\langle |\nabla v_I|^2\bigr\rangle ds + 
\frac{qd\delta}{\alpha} 
\sum_I \int_0^ t\bigl\langle |\nabla v_I|^2 \bigr\rangle ds + 
%\frac{qdc_\delta}{4\alpha} 
\frac{qdc_\delta}{\alpha} 
\sum_I \int_0^ t\bigl\langle v_I^2 \bigr\rangle ds.
\end{align*}
Also, by \eqref{e:si1}, we have
$$
\sum_I |S^1_I| \leq 
\left(\gamma+\frac{\delta}{4\gamma}\right)\sum_I \int_0^ t\bigl\langle |\nabla v_I|^2 \bigr\rangle ds
+ \frac{c_\delta}{4\gamma}\sum_I\int_0^t\langle v_I^2\rangle.
$$

Now, armed with the last two estimates, we sum both sides of \eqref{wI_id} over all $I$ with $|I|=2q$ to obtain
\begin{align*}
\frac{1}{2}\sum_I \bigl\langle v_I^2(t) \bigr\rangle & + \mu\int_0^t v_I^2 ds + \varkappa \int_0^ t  \sum_I \bigl\langle |\nabla v_I|^2 \bigr\rangle ds  \\
& \leq  \frac{1}{2} \sum_I \bigl\langle v_I^2(0)\bigr\rangle  + \biggl[
\frac{qdc_\delta}{\alpha} + 
\frac{c_\delta}{4\gamma}\biggr] \sum_I\int_0^t\langle v_I^2\rangle,
\end{align*}
where
$$
\varkappa:=\frac{\sigma^2}{2} - \gamma - \frac{\delta}{4\gamma} - 
4\alpha qd-\frac{qd\delta}{\alpha}.
$$
The maximum
$
\varkappa_*:=\max_{\alpha,\gamma>0}\varkappa= \frac{\sigma^2}{2}  - \sqrt{\delta} - 
4qd\sqrt{\delta}
$
is attained at 
$$\alpha=
\frac{\sqrt{\delta}}{2}, \quad \gamma=\frac{\sqrt{\delta}}{2}.
$$
For this choice of $\alpha$ and $\gamma$, we have 
$\varkappa_*=\frac{\sigma^2}{2}-\beta_{2q}\sqrt{\delta}$.
Since $\beta_{2q}\sqrt{\delta}<\frac{\sigma^2}{2}$ by assumption, we have $\varkappa_*>0$ and
\begin{equation*}
\frac{1}{2}\sum_I \bigl\langle v_I^2(t) \bigr\rangle  + \bigl(\mu-\hat{c}\,\bigr)\int_0^t v_I^2 ds  + \varkappa_* \int_0^ t  \sum_I \bigl\langle |\nabla v_I|^2 \bigr\rangle ds \leq  \frac{1}{2} \sum_I \bigl\langle v_I^2(0) \bigr\rangle,
\end{equation*}
where 
$\hat{c}:=\frac{2qdc_\delta}{\sqrt{\delta}} + \frac{c_\delta}{2\sqrt{\delta}}$. 
Thus, choosing  $\mu \geq \hat{c}$, 
we obtain
$$
\frac{1}{2}\sup_{\tau \in [0,t]}\sum_I \bigl\langle v_I^2(\tau) \bigr\rangle  + \varkappa_* \int_0^ t  \sum_I \bigl\langle |\nabla v_I|^2 \bigr\rangle ds \leq  \frac{1}{2} \sum_I \bigl\langle v_I^2(0) \bigr\rangle.
$$

\smallskip

\textit{Step 3}.~Recalling that $v_I=\mathbb E\bigl[\prod_{r \in I} \partial_{x_r}u \bigr]$, $v_I(0)=\prod_{r \in I}\partial_{x_r} f$, we obtain from the previous estimate:
\begin{align}
&\sup_{t \in J_T} \sum_{1 \leq k \leq d}\bigl\langle (\mathbb E(\partial_{x_k}u)^{2q})^2\bigr\rangle 
\leq c_1 \bigl\langle |\nabla f|^{2q} \bigr\rangle, 
 \label{ww1}\\
&\sum_{1 \leq k \leq d}\int_0^t \bigl\langle |\nabla \mathbb E(\partial_{x_k}u)^{2q}|^2\bigr\rangle ds 
\leq c_2 \bigl\langle |\nabla f|^{2q} \bigr\rangle,
\label{ww2}
\end{align}
for appropriate positive constants $c_1$, $c_2$. 
By the Sobolev embedding theorem,
$$
\int_0^t \bigl\langle (\nabla \mathbb E|\nabla u|^{2q})^2\bigr\rangle ds \geq 
c_3\int_0^t 
\bigl\langle (\mathbb E|\nabla u|^{2q})^{\frac{2d}{d-2}} \bigr\rangle^{\frac{d-2}{d}} ds,
$$
so \eqref{ww2}  yields
$$\|\mathbb E|\nabla u|^{2q}\|^2_{L^2(J_T,L^\frac{2d}{d-2})} \leq 
c_4\|\nabla f\|_{4q}^{4q},
$$
for appropriate constant $c_4>0$.

Interpolating between the last estimate, 
and \eqref{ww1}, that is, 
$\|E|\nabla u|^{2q}\|^2_{L^\infty(J_T,L^2)} \leq c_1 \|\nabla f\|_{4q}^{4q},$ 
we obtain \eqref{apr_est}.
\end{proof}

\section{Proof of Theorem \ref{thm1}}
 Recall that $\|\cdot\|_{p,\rho}$ denotes the norm in $L^p(\mathbb R^d,\rho dx)$, and $\langle \cdot ,\cdot \rangle_\rho$ the inner product in $L^2(\mathbb R^d,\rho dx)$.
We assume throughout this section that $b\in \mathbf{F}_\delta$ and $b_m$,
$m=1, 2, \dots$ are as in the beginning of the previous section.

\begin{lemma} \label{lem1} Let $b \in \mathbf{F}_\delta$, and let $b_m$ be as 
above. 
Then the following are true:

{\rm (\textit{i})} $\|b\sqrt{\rho}\|_2<\infty$.

{\rm (\textit{ii})} $\|b\sqrt{\rho}\mathbf{1}_{B^c(0,R+1)}\|_2 \downarrow 0$ as $R \rightarrow \infty$.

{\rm (\textit{iii})} $\langle \rho |b-b_m|^2\rangle \rightarrow 0$ as $m \rightarrow \infty$.
\end{lemma}
\begin{proof}
(\textit{i}) Using $b \in \mathbf{F}_\delta$, and applying \eqref{two_est} and 
$\langle \rho \rangle<\infty$, we have
$$
\|b\sqrt{\rho}\|_2^2 \leq \delta \|\nabla \sqrt{\rho}\|_2^2 +c_\delta\langle \rho \rangle <\infty.
$$

(\textit{ii}) 
For any $R\ge 1$, let $\eta_R$ be a $[0, 1]$-valued smooth function such that $\eta_R(x)=1$ if $|x|>R+1$; $\eta_R(x)=0$ if $|x| \leq R$; and $\sup_{R\ge 1}\|\nabla \eta_R\|_\infty\le C$.
Then
\begin{align*}
\|b\sqrt{\rho}\eta_R\|_2^2 \leq \delta \|\nabla [\sqrt{\rho}\eta_R]\|_2^2 + c_\delta\langle \rho\eta_R^2 \rangle.
\end{align*}
We have $\nabla [\sqrt{\rho} \eta_R]=\frac{1}{2\sqrt{\rho}}(\nabla \rho)\eta_R + \sqrt{\rho} \nabla \eta_R=:S_1+S_2$.
Using \eqref{two_est}, we have
$$
\|S_1\|_2^2 \leq C\langle \rho \eta^2_R \rangle \rightarrow 0 \quad \text{ as } R \rightarrow \infty.
$$
Next, we use $\sup_{R\ge 1}\|\nabla \eta_R\|_\infty\le C$ to get
\begin{align*} 
\|S_2\|_2^2 \leq  C (1+\kappa R^2)^{-\theta}  \langle \mathbf{1}_{B(0,R+1)-B(0,R)}\rangle 
= c_d C(1+\kappa R^2)^{-\theta}  R^d \rightarrow 0 \text{ as } R \rightarrow \infty 
\end{align*}
since $\theta>\frac{d}{2}$. This completes the proof of (\textit{ii}).

\smallskip

(\textit{iii}) This is a consequence of (\textit{ii}) and $b_m \rightarrow b$ in $L^2_{\loc}(\mathbb R^d)$.

The proof of Lemma \ref{lem1} is complete.
\end{proof}

\begin{lemma}
\label{cl1}
Let $\beta_2\sqrt{\delta}<\frac{\sigma^2}{2}$, 
$f \in C_c^\infty$ 
and $u_m$ be the strong solution to \eqref{eq5}.
Provided that $\kappa>0$ in the definition of  $\rho$ 
is chosen sufficiently small, 
there exists $\mu\big(\delta, c_\delta\big) \geq 0$ 
such that for any $\mu\ge \mu\big(\delta, c_\delta\big)$, 
$$\lim_{n, m\to\infty}\sup_{t \in J_T}\|\mathbb E|u_n(t)-u_m(t)|^2\|_{2,\rho}=0.$$
\end{lemma}
\begin{proof}
Set
$$
g \equiv g_{n,m}:=u_n-u_m, \quad n,m=1,2,\dots,
$$
then
$$
g(t)  + \mu\int_0^t g ds + \int_0^t b_m \cdot \nabla g ds + \int_0^t (b_n-b_m) \cdot \nabla u_m ds + \sigma \int_0^t \nabla g dB_s - \frac{\sigma^2}{2}\int_0^t \Delta g ds=0.
$$
Applying It\^{o}'s formula, we obtain
$$
g^2(t) = 
-2\mu\int_0^t g^2 ds - \int_0^t b_m \cdot \nabla g^2 ds - 2\int_0^t g (b_n-b_m) 
\cdot \nabla u_m ds -\sigma\int_0^t \nabla g^2 dB_s + \frac{\sigma^2}{2}\int_0^t \Delta g^2 ds,
$$
so denoting $h:=\mathbb E[g^2]$ we arrive at
$$
\partial_t h + 
2\mu h - \frac{\sigma^2}{2} \Delta h + b_m \cdot \nabla h + 2(b_n-b_m) 
\cdot \mathbb E[g\nabla u_m]=0, \quad h(0)=0.
$$
Multiplying this equation by $\rho h$ and integrating by parts, we obtain
\begin{align}
\label{id9}
\frac{1}{2} \|h(t)\|_{2,\rho}^2 & + 
2\mu\int_0^t \|h\|^2_{2,\rho}ds 
+  \frac{\sigma^2}{2} \int_0^t \|\nabla h\|_{2,\rho}^2 ds +\frac{\sigma^2}{2}\int_0^t \langle (\nabla \rho)h, \nabla h \rangle \\
& + \int_0^t \langle b_m \cdot \nabla h,h\rangle_\rho ds  + 
2\int_0^t \langle h (b_n-b_m) \cdot \mathbb E[g\nabla u_m] \rangle_\rho ds
=0. \notag
\end{align}

Since our assumption on $\delta$ is a strict inequality, using \eqref{two_est} and selecting $\kappa$ sufficiently small, we 
can and will ignore in what follows the terms containing $\nabla \rho$.

Applying the quadratic inequality and using $b_m \in \mathbf{F}_\delta$, we obtain (cf.\,the proof of \eqref{e1})
$$
\frac{\sigma^2}{2} \int_0^t \|\nabla h\|_{2,\rho}^2 ds   + \int_0^t \langle b_m \cdot \nabla h,h\rangle_\rho ds \geq \bigg(\frac{\sigma^2}{2}-\sqrt{\delta}\bigg)  \int_0^t \|\nabla h\|_{2,\rho}^2 ds - \frac{c_\delta}{4\sqrt{\delta}}\int_0^t \|h\|_{2,\rho}^2ds,
$$ 
where $\frac{\sigma^2}{2}-\sqrt{\delta}>0$ by the assumption on $\delta$.

We obtain from \eqref{id9}:
\begin{align*}
\frac12\sup_{\tau \in [0,t]} \|h(\tau)\|_{2,\rho}^2 
& + \bigg(\frac{\sigma^2}{2}-\sqrt{\delta}\bigg)\int_0^t \|\nabla h(s)\|_{2,\rho}^2 ds + 
\biggl[2\mu- \frac{c_\delta}{4\sqrt{\delta}}\biggr] 
\int_0^t \|h\|_{2,\rho}^2ds\\
& \leq 
2\int_0^t \langle h |b_n-b_m| \cdot \mathbb E[|g\nabla u_m|]\rangle_\rho  ds.
\end{align*}
Select $\mu \geq \frac{c_\delta}{4\sqrt{\delta}}$. 
Then the previous estimate yields
$$
\frac{1}{2}\sup_{\tau \in [0,t]} \|h(\tau)\|_{2,\rho}^2  \leq  
2\int_0^t \langle h |b_n-b_m| \cdot \mathbb E[|g\nabla u_m|]\rangle_\rho  ds,
$$
so it remains to show that 
$$
\int_0^t \langle h |b_n-b_m| \cdot \mathbb E[|g\nabla u_m|]\rangle_\rho  ds \rightarrow 0 \quad \text{ as } n,m \rightarrow \infty.
$$
We estimate 
\begin{align*}
\langle h|b_n-b_m| \cdot \mathbb E[|g\nabla u_m|]\rangle_\rho & \leq \langle |b_n-b_m|h(\mathbb E[g^2])^{\frac{1}{2}}(\mathbb E[|\nabla u_m|^2])^{\frac{1}{2}}\rangle_\rho \equiv \langle |b_n-b_m| h^{\frac{3}{2}}(\mathbb E[|\nabla u_m|^2])^{\frac{1}{2}}\rangle_\rho \\
& \leq \langle |b_n-b_m|^2 \rangle_\rho^{\frac{1}{2}} \langle h^3 \mathbb E[|\nabla u_m|^2] \rangle_\rho^{\frac{1}{2}} \leq \langle |b_n-b_m|^2 \rangle_\rho^{\frac{1}{2}} \langle h^3 \mathbb E[|\nabla u_m|^2] \rangle^{\frac{1}{2}} \\
& \leq \langle |b_n-b_m|^2 \rangle_\rho^{\frac{1}{2}} \langle h^6\rangle^{\frac{1}{4}} \langle (\mathbb E[|\nabla u_m|^2])^2 \rangle^{\frac{1}{4}} \\
& (\text{we apply Proposition \ref{apr_prop}, 
and \eqref{ww1} with $q=1$}) \\
& \leq 
c\langle |b_n-b_m|^2 \rangle_\rho^{\frac{1}{2}} \|f\|_{12}^{3} \|\nabla f\|_4 \\
&  (\text{we apply Lemma \ref{lem1}(\textit{iii})}) \\
& \rightarrow 0 
\quad \text{ as }n,m \rightarrow \infty.
\end{align*}
The proof of Lemma \ref{cl1} is complete.
\end{proof}

Lemma \ref{cl1} allows to prove that $\{u_m\}$ is a Cauchy sequence in $L^\infty(J_T,L^2(\Omega,L^2_\rho))$.

\begin{lemma}\label{l:thrm1-iii} 
Let $\beta_2\sqrt{\delta}<\frac{\sigma^2}{2}$, 
$f \in C_c^\infty$ 
and $u_m$ be the strong solution to \eqref{eq5}.
Provided that $\kappa>0$ in the definition of  $\rho$ 
is chosen sufficiently small, it holds that 
$u_m$ converges in $L^2(\Omega,L_\rho^2)$ to a process $u$, 
uniformly in $t \in J_T$.
\end{lemma}

\begin{proof}
%By Lemma \ref{cl1},
Let $\kappa$ be small enough and  $\mu$ greater than or equal to the
$\mu(\delta, c_\delta)$. Let $\mu\ge \mu(\delta, c_\delta)$. Then by
Lemma \ref{cl1},
$$
\sup_{t \in J_T}\mathbb E\|(u_n(t)-u_m(t))\|_{2,\rho}^2 \leq \langle \rho \rangle^{\frac{1}{2}} \sup_{t \in J_T}\|\mathbb E|u_n(t)-u_m(t)|^2\|_{2,\rho} \rightarrow 0
$$
as $m$, $n \rightarrow \infty$. Thus, we can define 
$$
u(t):=s{\mbox-}L^2(\Omega,L^2_\rho)\mbox{-}\lim_{m}u_m(t) \quad \text{ uniformly in } t \in J_T.
$$
The proof is complete
\end{proof}

We are in position to give
the proof of Theorem \ref{thm1}.

\medskip

\begin{proof}[Proof of Theorem \ref{thm1}] It suffices to carry out the proof for 
$f \in C_c^\infty$, and then use a density argument.

It follows from the assumption $\sqrt{\delta}<\frac{\sigma^2}{2\beta_2}$ that $p\ge 2$ is in the interval $(p_c, \infty)$, $p_c=\big(1-\frac{\sqrt{\delta}}{\sigma^2}\big)^{-1}$. (Indeed, $p_c<2$ if and only if $\sqrt{\delta}<\frac{\sigma^2}{2}$. In particular, $p_c<2$ if $\sqrt{\delta}<\frac{\sigma^2}{2\beta_2}$ since $\beta_2>1$.)
Let $\mu(\delta, c_\delta, p)$ be the constant from Proposition \ref{apr_prop}. Assume that $\mu\ge \mu(\delta, c_\delta, p)$. Then the conclusions of Proposition \ref{apr_prop} are valid.

\smallskip

We prove (\textit{i}) first.  We do this in two steps.

\textit{Step 1.}
Selecting $\kappa$ sufficiently small so that Lemma \ref{l:thrm1-iii}  applies, we obtain
that
$u_m$ converges in $L^2(\Omega,L_\rho^2)$ to a process $u$, 
uniformly in $t \in J_T$.
Thus $u \in L^\infty(J_T,L^2_{\loc}(\mathbb R^d,L^2(\Omega))$, and we have for all $t \in J_T$,
\begin{equation}
\label{c1_}
u_m \rightarrow u \quad \text{ in } L^\infty(J_T,L^2(\Omega,L_\rho^2)),
\end{equation}
which yields
\begin{equation}
\label{c2_}
\int_0^t u_m ds \rightarrow \int_0^t u ds \quad \text{ in } L^2(\Omega,L^2_\rho);
\end{equation}
the latter, \eqref{e2} and a standard weak compactness argument yield
\begin{equation}
\label{c3_}
\nabla \int_0^t u_m ds \rightarrow \nabla \int_0^t u ds \quad \text{ weakly in } L^2(\Omega,L_\rho^2(\mathbb R^d,\mathbb R^d)).
\end{equation}

\medskip

\textit{Step 2.}\,Given a test function $\varphi \in C_c^\infty$, we multiply \eqref{eq5} by $\rho\varphi$, integrate and write (we take $\mu=0$ to shorten calculations)
\begin{align}
\langle u_m(t)-u(t),\rho\varphi \rangle  + \langle u(t),\rho\varphi\rangle - \langle f,\rho\varphi\rangle 
& =   - \bigl\langle (b_m-b)\cdot \nabla \int_0^t u_mds, \rho\varphi \bigr\rangle - \bigl\langle b\cdot\nabla \int_0^t u_mds, \rho \varphi \bigr\rangle \notag \\ 
& + \sigma\bigl\langle \int_0^t(u_m-u)dB_s,\nabla \rho\varphi \bigr\rangle + \sigma\bigl\langle \int_0^tudB_s,\nabla \rho \varphi \bigr\rangle \label{g2} \\
& -\frac{\sigma^2}{2}\bigl\langle \nabla \int_0^t(u_m-u)ds, \nabla \rho \varphi \bigr\rangle - \frac{\sigma^2}{2}\bigl\langle \nabla \int_0^t uds, \nabla \rho \varphi \bigr\rangle. \notag
\end{align}

Let us now note the following.
In view of \eqref{c1_} and \eqref{c3_},
$
\langle u_m(t)-u(t),\rho\varphi \rangle \equiv \langle u_m(t)-u(t),\varphi \rangle_\rho \rightarrow 0 \text{ in } L^2(\Omega)$. Similarly, 
using \eqref{c3_} and \eqref{two_est},
\begin{equation}
\tag{a}
\bigl\langle \nabla \int_0^t(u_m-u)ds, \nabla \rho \varphi \bigr\rangle  \rightarrow 0 \text{ weakly in } L^2(\Omega),
\end{equation}
and, since $\varphi |b| \in L^2_\rho$ (using that $\varphi$ has compact support),
\begin{equation}
\tag{b}
\bigl\langle b\cdot\nabla \int_0^t u_m ds, \rho\varphi \bigr\rangle \rightarrow \bigl\langle b\cdot\nabla \int_0^t uds, \rho\varphi \bigr\rangle \text{ weakly in } L^2(\Omega).
\end{equation}
By \eqref{e2}, $\|\nabla\int_0^t u_mds\|_{L^2(\Omega,L^2_\rho)} \leq 
c_1$ with $c_1<\infty$ independent of $m$, 
and 
 $\varphi |b_m-b_n| \rightarrow 0$ in $L^2_\rho$ (in fact, in $L^2$). Thus
\begin{equation}
\tag{c}
\bigl\langle (b_m-b)\cdot\nabla \int_0^t u_mds, \rho\varphi \bigr\rangle \rightarrow 0 \text{ in } L^2(\Omega).
\end{equation}
Finally, let us show that
\begin{equation}
\tag{d}
\bigl\langle \int_0^t(u_m-u)dB_s,\nabla \rho\varphi \bigr\rangle  \rightarrow 0 \text{ in } L^2(\Omega).
\end{equation}
Indeed, using It\^{o}'s isometry, we have using \eqref{two_est}
\begin{align*}
\mathbb E\biggl|\bigl\langle \int_0^t(u_m-u)dB_s,\nabla \rho \varphi \bigr\rangle\biggr|^2 & \leq 
c_2\mathbb E\langle \big|\int_0^t(u_m-u)dB_s\big|^2\rangle_\rho 
\langle |\varphi|^2\rangle_\rho \\
& = c_3 \langle \mathbb E \int_0^t (u_m-u)^2 ds\rangle_\rho 
\rightarrow 0 \quad \text{ by \eqref{c1_}}.
\end{align*}
The convergence (d) follows.

Thus, using (a)-(d), we can pass to the $L^2(\Omega)$-weak limit in \eqref{g2} as $m \rightarrow \infty$, obtaining that $u$ satisfies \eqref{eq__} (with test functions $\varphi \rho$ which, clearly, exhaust $C_c^\infty$).

The estimates in \eqref{grad_reg}, \eqref{grad_reg_} now follow from Proposition \ref{apr_prop}.

\medskip

The last assertion (\textit{ii}) is Lemma \ref{l:thrm1-iii} proved above.

The proof of Theorem \ref{thm1} is complete.
\end{proof}

\section{Proof of Theorem \ref{thm1_reg}}

\begin{proof}[Proof of Theorem \ref{thm1_reg}]  Part (a) follows from Theorem 
\ref{thm1}(\textit{i}).
The last assertion, \eqref{grad_reg2},
follows from Proposition \ref{apr_prop-2} and Lemma \ref{l:thrm1-iii}.
So we only need to prove part (b).

Since the weak-$L^2(J_T \times \Omega)$ limit of any sequence of $(\mathcal F_t)$-progressively measurable processes on $J_T$  remains $(\mathcal F_t)$-progressively measurable
and $t \mapsto \langle u_m(t), \varphi\rangle$ is $(\mathcal F_t)$-progressively measurable for every $m$,
in view of \eqref{c2_}, the process $t \mapsto \langle u(t), \varphi\rangle$ is $(\mathcal F_t)$-progressively measurable as well.
The proof of \eqref{eq___} follows closely the proof of \eqref{eq__} above except that now, instead of \eqref{e2}, we appeal to the Sobolev regularity estimate \eqref{grad_reg3} with $q=1$.

The existence of a continuous $(\mathcal F_t)$-semi-martingale modification of $t \mapsto \langle u(t), \varphi\rangle$ is a consequence of the identity \eqref{eq___}.

The proof of Theorem \ref{thm1_reg} is complete.
\end{proof}

\section{Proof of Theorem \ref{thm2} (weak uniqueness)}

\label{unique_sect}

The fact that \eqref{eq2} has at least one weak solution was proved in Theorem \ref{thm1_reg}. We now prove its uniqueness.
We adopt the argument of \cite[Sect.\,3]{BFGM}. 
We will need the following definitions and results. 
Let us fix a version of the Brownian motion $B_t$ having continuous trajectories $B_t(\omega)$ for every $\omega \in \Omega.$

\smallskip

%RS Do we need $\kappa$ small enough and $\mu$ larger than or equal to some $\mu(\delta, c_\delta)$ in the following lemma?
\begin{lemma}
\label{lem1_w}
Let $b \in \mathbf{F}_\delta$ with 
$\sqrt{\delta}<\frac{\sigma^2}{2\beta_2}$ and $f\in W^{1, 4}$.
Let $u=u(t,x,\omega)$ be a weak solution to \eqref{eq2}. Then for a.e. $\omega \in \Omega$,
$$
\tilde{u}^\omega(t,x):=u(t,x+\sigma B_t(\omega),\omega)
$$
is a weak solution to the Cauchy problem 
\begin{equation}
\label{random_pde1}
\partial_t \tilde{u}^\omega + \mu \tilde{u}^\omega + \tilde{b}^\omega \cdot \nabla \tilde{u}^\omega=0, \quad \tilde{u}^\omega|_{t=0}=f, \quad \text{ where }\; \tilde{b}^\omega(t,x):=b(x+\sigma B_t(\omega)),
\end{equation}
that is, the following are true:
\smallskip

{\rm 1)} 
$\tilde{u}^\omega \in L^\infty(J_T,W^{1,2}_{\rho})$;
\smallskip

{\rm 2)} for every 
$\psi \in C^1(J_T,C_c^\infty)$, 
the function $t \mapsto \langle \tilde{u}^\omega(t),\psi(t)\rangle $ has a continuous representative, i.e.\,a continuous function which coincides with $t \mapsto \langle \tilde{u}^\omega(t),\psi(t)\rangle $ for a.e.\,$t \in J_T$;

{\rm 3)} for every 
$\psi \in C^1(J_T,C_c^\infty)$,
this continuous representative of $t \mapsto \langle \tilde{u}^\omega(t),\psi(t)\rangle $ satisfies for every $t \in J_T$,
$$
\langle \tilde{u}^\omega(t),\psi(t)\rangle=\langle f,\psi(0)\rangle + \mu \int_0^t \langle \tilde{u}^\omega(s), \psi(s) \rangle ds + \int_0^t \langle \tilde{u}^\omega(s), \partial_s \psi(s)\rangle ds - \int_0^t \langle \nabla \tilde{u}^\omega(s), \tilde{b}^\omega(s) \psi(s)\rangle ds.
$$ 
\end{lemma}

The proof of Lemma \ref{lem1_w} follows closely the proof of \cite[Prop.\,3.4]{BFGM} (taking into account the definition of 
the weak solution to \eqref{eq2}) and we omit the details.

\medskip

Consider the terminal value problem
\begin{equation}
\label{v_eq}
dv_m + \mu v_m dt + \nabla \cdot (b_m v_m) dt + \sigma \nabla v_m \circ dB_t=0, \quad t \in [0,t_*], \quad v_m|_{t=t_*}=v_0 \in C_c^\infty,
\end{equation}
where $b_m \in C_c^\infty(\mathbb R^d,\mathbb R^d)$ ($m=1, 2, \dots$)
(since $b_m$ are bounded and smooth, we have strong existence and uniqueness for this equation).

The following is an analogue of \cite[Cor.\,3.8]{BFGM}.

\begin{lemma}
\label{lem2_w}
$
\tilde{v}^\omega_m(t,x):=v_m(t,x+\sigma B_t(\omega))
$
satisfies, for a.e.\,$\omega \in \Omega$, $\tilde{v}_m^\omega \in C^1([0,t_*],C^\infty_c)$ and
$$
\partial_t \tilde{v}^\omega_m  + \mu \tilde{v}^\omega_m + \nabla \cdot (b^\omega_m \tilde{v}^\omega_m)=0, \quad \tilde{v}^\omega_m(t_*,x)=v_0(x+\sigma B_{t_*}(\omega)).
$$
\end{lemma}

We will also need

\begin{lemma}
\label{step3}
Let $\sqrt{\delta}<\frac{\sigma^2}{6}$.
There exist a constant $\mu(c_\delta) \geq 0$ and a sufficiently small $\kappa>0$ (in the definition of $\rho$) such that 
$$
\sup_{t \in J_T}\|\rho^{-1}\mathbb E[v_m^2(t)]\|_2 \leq \|\rho^{-1}v_0\|_{4}^2, \quad \mu \geq 
\mu(c_\delta), m=1, 2, \dots
$$
where $v_m$ is the strong solution to \eqref{v_eq}. 
\end{lemma}
\begin{proof}
Without loss of generality, we will carry out the proof for the forward equation, 
and will drop the subscript $m$ from $b_m$. Set $w:=\mathbb E[v^2]$. Arguing as in the proof of Proposition \ref{apr_prop}, we obtain that $w$ satisfies
\begin{equation}
\label{w_eq}
\partial_t w + 2\mu w - \frac{\sigma^2}{2}\Delta w -2 \nabla \cdot (bw) + b \cdot \nabla w=0, \quad w(0)=v_0^2.
\end{equation}

We first carry out the proof for $\rho \equiv 1$. 
Multiplying the previous equation by $w$ and integrating, we obtain
$$
\frac{1}{2}\partial_t \langle |w|^2\rangle + 2\mu\langle |w|^2\rangle +
\frac{\sigma^2}{2}\langle |\nabla w|^2\rangle + 3\langle \nabla w,bw\rangle=0.
$$
Applying the quadratic inequality and the form-boundedness condition $b \in \mathbf{F}_\delta$, 
we get that, for any $\gamma>0$, 
\begin{align*}
\frac{1}{2}\partial_t \langle |w|^2\rangle + (2\mu-3\gamma c_\delta)\langle |w|^2\rangle + \biggl[\frac{\sigma^2}{2}-3(\gamma \delta + \frac{1}{4\gamma}) \biggr]\langle |\nabla w|^2\rangle  \leq 0, 
\end{align*}
and so, selecting  $\mu(c_\delta):=\frac{3}{2}\gamma c_\delta$ and $\mu \geq \mu(c_\delta)$, we obtain
$$
\frac{1}{2}\langle |w(t)|^2\rangle +  \biggl[\frac{\sigma^2}{2}-3(\gamma \delta + \frac{1}{4\gamma}) \biggr]\int_0^t \langle |\nabla w|^2\rangle ds \leq  \frac{1}{2}\langle |v_0|^{4}\rangle.
$$
Upon maximizing the coefficient in the square brackets in $\gamma$ (thus, selecting $\gamma=\frac{1}{2\sqrt{\delta}}$), we obtain that the coefficient is positive since  $\sqrt{\delta}<\frac{\sigma^2}{6}$.
In particular, 
it follows that $\sup_{t \in J_T}\|\mathbb E[v_m^2(t)]\|_2 \leq \|v_0\|_{4}^2$.

\medskip

 In presence of $\rho^{-1}$,
we argue as above but get new terms containing $\nabla \rho^{-1}$, which we bound appealing to the estimate
\begin{align*}
|\nabla \rho^{-1}| =\left|\frac{\nabla \rho}{\rho^2}\right|  \leq \theta\sqrt{\kappa}\rho^{-1} \quad \text{(by \eqref{two_est})},
\end{align*}
with $\kappa$ selected sufficiently small. (Note that to justify $\|\rho^{-1}\mathbb E[v_m^2(t)]\|_2<\infty$ we can appeal to qualitative Gaussian upper bound on the heat kernel of \eqref{w_eq}.)
\end{proof}

Let us note that the assumption of the theorem 
$\beta_2\sqrt{\delta}<\frac{\sigma^2}{2}$ implies
$\sqrt{\delta}<\frac{\sigma^2}{6}$.

\medskip

We are now in position to complete the proof of Theorem \ref{thm2}.

\begin{proof}[Proof of Theorem \ref{thm2}]
Let $\mu$ and $\kappa$ be as in Lemma \ref{step3}.
In view of the linearity of the stochastic transport equation, it suffices to show that a weak solution $u$ to \eqref{eq2} with initial condition $u(0)=0$ must be identically zero for all $t \in J_T$. In view of Lemma \ref{lem1_w}, it suffices to show that $\tilde{u}^\omega$ corresponding to $u$ is identically zero a.s.

Let $v_0 \in C_c^\infty$.
It follows from Lemma \ref{lem2_w} that,  for a.e.\,$\omega \in \Omega$, $\tilde{v}^\omega(s)\in C^1(J_T, C_c^\infty)$.
Thus by  Lemma \ref{lem1_w}, 
for a.e.\,$\omega \in \Omega$ with $\psi(s):=\tilde{v}^\omega(s)$, for all $0<t_* \leq T$,
\begin{align*}
\label{u_omega} \tag{$\bullet$} 
&\langle \tilde{u}^\omega(t_*),v_0(\cdot + \sigma B_{t_*}(\omega))\rangle \\
& = \mu \int_0^{t_*} \langle \tilde{u}^\omega(s), \tilde{v}_m^\omega(s)\rangle d + 
\int_0^{t_*} \langle \tilde{u}^\omega(s), \partial_s \tilde{v}_m^\omega(s)\rangle ds - \int_0^{t_*} \langle \nabla \tilde{u}^\omega(s), \tilde{b}^\omega(s)\tilde{v}_m^\omega(s)\rangle ds \\
& = \int_0^{t_*} \langle \nabla \tilde{u}^\omega, 
(\tilde{b}_m^\omega(s)-\tilde{b}^\omega(s))
\tilde{v}_m^\omega\rangle ds =: I.
\end{align*}

Step 1. Let us first show that
\begin{equation}
\label{conv_E}
\tag{$\bullet\bullet$}
\mathbb E \left| \int_0^{t_*} \langle \nabla u, (b-b_m)v_mn \rangle ds \right| \rightarrow 0 \quad \text{as $m \uparrow \infty$}.
\end{equation}
We have
\begin{align*}
& \mathbb E \left| \int_0^{t_*} \langle \nabla u, (b-b_m)v_m \rangle ds \right| \leq \int_0^{t_*} \big\langle |b-b_m| \mathbb E\big[|\nabla u|^2\big]^{\frac{1}{2}}\mathbb E\big[|v_m|^2\big]^{\frac{1}{2}}\big\rangle ds \\
& \leq \biggl(\int_0^{t_*} \big\langle \rho|b-b_m|^2\big\rangle ds \biggr)^{\frac{1}{2}} \biggl(\int_0^{t_*} \big\langle  (\mathbb E\big[|\nabla u|^2\big])^2 \big\rangle ds \biggr)
^{\frac{1}{4}} 
\biggl(\int_0^{t_*} \big\langle\rho^{-2}(\mathbb E\big[|v_m|^2\big])^2\big\rangle ds \biggr)
^{\frac{1}{4}}.
\end{align*}
The first integral converges to $0$ as $m \uparrow \infty$ by Lemma \ref{lem1}(\textit{iii}), the second integral is finite by the definition of weak solution before Theorem \ref{thm2}, and the third integral is bounded from above uniformly in $m$ by $\sqrt{t_*}\|\rho^{-1}v_0\|^2_{4}<\infty$, see Lemma \ref{step3}. Thus, \eqref{conv_E} follows.

Step 2.~By Step 1, there exists a 
subset $\Omega_{t_*,v_0} \subset \Omega$  of probability 1
and a sequence $m_k \uparrow \infty$ such that for every 
$\omega \in \Omega_{t_*,v_0}$,
$$
\int_0^{t_*} \langle \nabla u, (b-b_{m_k})v_{m_k} \rangle ds \rightarrow 0 \quad \text{ as } m_k \uparrow \infty.
$$
Making the change of variable $x \mapsto x + \sigma B_t(\omega)$ and using the fact 
that $c_{t_*, w}^{-1} \rho(\cdot)\le \rho(\cdot + \sigma B_t(\omega)) \leq c_{t_*, w} \rho(\cdot)$ for some constant $c_{t_*, w}>1$
we obtain that for every 
$\omega \in \Omega_{t_*,v_0}$,
$$
I \rightarrow 0 \quad \text{ as } m_k \uparrow \infty.
$$

Fix a countable dense subset $D$ of $C_c^\infty(\mathbb R^d)$ and define
$$
\tilde{\Omega}:=\bigcap_{t_* \in [0,T] \cap \mathbb Q, v_0 \in D}\, \Omega_{t_*,v_0},
$$
a full measure set in $\Omega$. Applying the diagonal argument (and so passing to a subsequence of $\{\varepsilon_k\}$), we obtain by \eqref{u_omega} and Step 2 that for every $\omega \in \tilde{\Omega}$, $\tilde{u}^\omega(t)=0$ for all $t \in [0,T] \cap \mathbb Q$. Since $t \mapsto \langle \tilde{u}^\omega(t),\varphi\rangle $, $\varphi \in C_c^\infty(\mathbb R^d)$ 
is continuous, we obtain that $\tilde{u}^\omega(t)=0$ for all $t \in [0,T]$ for all $\omega \in \tilde{\Omega}$, as needed. 

The proof of Theorem \ref{thm2} is complete.
\end{proof}

\end{document}